\documentclass[11pt]{article}
\usepackage{"commands"}
\usepackage{"letters"}
\usepackage[T1]{fontenc}
\usepackage{verbatim}
\usepackage[tbtags]{amsmath}
\usepackage{amssymb,bbm}
\usepackage{amsthm}
\usepackage[usenames,dvipsnames,svgnames,table]{xcolor}
\usepackage{subcaption}
\usepackage[nottoc]{tocbibind}
\usepackage{multirow}
\usepackage{hyperref}

\usepackage{tikz}
\usepackage{authblk}
\usetikzlibrary{arrows}
\usetikzlibrary{positioning}
\usetikzlibrary{calc}
\usepackage{float}
\usetikzlibrary{shapes,snakes}
\usepackage{setspace}
\usepackage{algorithm}
\usepackage[noend]{algpseudocode}
\usepackage{longtable}
\usepackage[authoryear, round]{natbib}
\makeatletter
\newcommand{\algmargin}{\the\ALG@thistlm}
\makeatother

\algnewcommand{\parState}[1]{\State%
  \parbox[t]{\dimexpr\linewidth-\algmargin}{\strut #1\strut}}

\theoremstyle{remark}
\newtheorem*{remark}{Remark}

\theoremstyle{definition}

\theoremstyle{lemma}

\theoremstyle{theorem}

\theoremstyle{proposition}
\newtheorem{proposition}{Proposition}[section]

\providecommand{\keywords}[1]{\textit{Keywords:} #1}

\usepackage[margin=0.75in]{geometry}

\newcommand\score[2]{
\pgfmathsetmacro\pgfxa{#1+1}
\tikzstyle{scorestars}=[star, star points=5, star point ratio=2.25, draw,inner sep=1.3pt,anchor=outer point 3]
  \begin{tikzpicture}[baseline]
    \foreach \i in {1,...,#2} {
    \pgfmathparse{(\i<=#1?"black":"gray")}
    \edef\starcolor{\pgfmathresult}
    \draw (\i*1.75ex,0) node[name=star\i,scorestars,fill=\starcolor]  {};
   }
  \end{tikzpicture}
}

\usepackage{accents}

\newcommand\munderline[1]{%
  \underaccent{\bar}{#1}}
\newcommand{\taubar}{\bar{\tau}}
\newcommand{\taulow}{\munderline{\tau}}

\title{A multi-vehicle covering tour problem with speed optimization}
\author[1,${*}$]{Joshua T. Margolis}
\author[1]{Yongjia Song}
\author[1]{Scott J. Mason}
\affil[1]{Department of Industrial Engineering, Clemson University, Clemson, SC 29634, USA}
\date{}

\begin{document}

\maketitle
\begin{flushleft}
\text{*}\ {Corresponding author}\\
{E\--mail addresses: jtmargo@clemson.edu (J.T. Margolis), yongjis@clemson.edu (Y. Song),  mason@clemson.edu (S.J. Mason)}
\end{flushleft}

\begin{abstract}
{The multi-vehicle covering tour problem with time windows (MCTPTW) aims to construct a set of maximal coverage routes for a fleet of vehicles that serve (observe) a secondary set of sites given a fixed time schedule, coverage requirements, and energy restrictions. The problem is formulated as a mixed-integer second-order cone programming (MISOCP) model and is an extension of both the multi-covering tour problem and the vehicle routing problem with time windows under energy constraints. 
Further, we study a special case of the proposed model and develop a labeling algorithm to solve its Lagrangian relaxation problem, which exploits the combinatorial structure exhibited by an optimal solution to the Lagrangian relaxation.}
\end{abstract}
\keywords{vehicle routing; covering tour; Lagrangian relaxation; labeling algorithm}

\section{Introduction}
\label{sec:intro}

Unmanned (or autonomous) vehicles are aerial, terrestrial, and aquatic vehicles that operate without human intervention; they are becoming more prevalent in both commercial and governmental applications. As the technology surrounding them continues to advance, they will continue to be utilized in new and innovative ways. In particular, creating autonomous vehicles capable of adapting to real-time situations has recently become a priority in multiple industries. Rideshare companies (e.g., Uber and Lyft), technology companies (e.g., Apple and Google), and car manufacturers (e.g., Ford and General Motors) are all introducing self-driving cars capable of operating without human assistance. 

Research surrounding unmanned vehicles has highlighted the increasing desire to utilize them for a variety of purposes 
as they have many advantages: 1) they can be used to assist or even replace humans performing hazardous assignments; 2) they have increased maneuverability and deployability; and 3) they have low operation and maintenance costs \citep{Liu2016}. For instance, unmanned aerial vehicles (UAVs) can be deployed for agricultural purposes \citep{Faical2014, Herwitz2004, Tokekar2016}, traffic monitoring and management \citep{Chow2016, Kanistras2015}, forest fire detection \citep{Casbeer2006}, product delivery \citep{Coelho2017,  Poikonen2017}, and 
disaster relief efforts and humanitarian causes \citep{Chowdhury2017, Murphy2008}. For an extensive survey of additional civil UAV applications, we refer the reader to 
\citet{Otto2018}. In regards to military environments, numerous research papers focus on routing unmanned vehicles for reconnaissance \citep{Cao2017, Chung2011, Svec2014}, target engagement \citep{Shetty2008}, and surveillance missions \citep{Wallar2015}. 


In this paper, we study unmanned vehicle routing for target surveillance with specific vehicle characteristics (e.g., battery life, payload, speed, etc.) which can be considered as a variant of the vehicle routing problem (VRP). First studied by 
\citet{Dantzig1959}, the VRP aims to route a collection of vehicles at minimal distance through a set of locations such that each vehicle begins and ends its journey at a fixed depot. Many survey papers exist that provide excellent summaries of the research surrounding the VRP \citep{Cordeau2007, Laporte2007, Potvin2009, Toth2002}. Variants include the capacitated VRP \citep{Dantzig1959, Fukasawa2006,Lysgaard2004,Xu2010}, the VRP with pickup and delivery \citep{Parragh2008a, Parragh2008b}, and the VRP with time-windows (VRPTW) \citep{Braysy2005a, Braysy2005b, Kallehauge2005, Solomon1987}. Exact solution methodologies for the VRPTW have been explored previously in the literature, including 
1) branch\--and\--cut algorithms \citep{Bard2002} and 2) branch\--cut\--and\--price algorithms \citep{Desaulniers2010, Ropke2009}.
For a more detailed summary, we refer the reader to recent VRP review papers~\citep{Baldacci2012, Cordeau2000, Kallehauge2008}.


Specifically, we focus on the multi-vehicle covering tour problem with time windows (MCTPTW), which aims to construct a set of maximal coverage routes for a fleet of vehicles that serve (observe) a secondary set of sites given fixed time window constraints, coverage requirements, and energy restrictions. Vehicles are routed through coverage areas by visiting a set of waypoints (locations); we utilize the \textit{covering tour problem} (CTP) as a basis for our model. The CTP seeks to find a minimum\--length Hamiltonian cycle along a set of vertices that passes within a predetermined distance of a secondary set of nodes so that these nodes are ``covered'' by the cycle~\citep{Current1989,Gendreau1997}.~\citet{Hachicha2000} extend the CTP to include multiple routes (MCTP) and solve the problem heuristically.~\citet{Jozefowiez2015} was the first to solve the MCTP exactly via a branch\--and\--price algorithm, where the resulting subproblem is a variant of the \textit{traveling salesman problem with profits} and reduces to a \textit{ring star problem} that is solved using a branch\--and\--cut algorithm. 
\citet{Ha2013b} analyze a variant of the problem with no tour length constraints and develop a branch\--and\--cut algorithm and a two-phase metaheuristic derived from evolutionary local search to solve it.~\citet{Tricoire2012} introduce stochastic demand to the MCTP and optimize with respect to multiple objectives (i.e., cost and expected uncovered demand); they reformulate the multi-objective problem using the $\epsilon$\--constraint method and solve it via a branch\--and\--cut approach. 


Unlike the MCTP's node\--based definition of coverage, 
the close-enough arc routing problem (CEARP) 
defines coverage based upon the proximity of traversed arcs 
in relation to the location of interest \citep{Shuttleworth2008}; we implement both definitions in our study. Since the CEARP is a generalization of the \textit{directed rural postman problem} (DRPP), \citet{Shuttleworth2008} utilize DRPP heuristics as a solution methodology after identifying arc subsets. \citet{Ha2013a} solve the CEARP exactly using a mixed\--integer programming formulation and cutting-plane approaches. Both \citet{Ha2013a} and \citet{Shuttleworth2008} examine the problem from the perspective of meter reading. Furthermore, \citet{Coutinho2016} model the CEARP as a second\--order cone and solve it exactly using a branch-and-bound algorithm; they test their approach on two- and three-dimensional instances. Coverage constraints, moreover, can be time-dependent; thus, it is important to include the speed at which an arc is traversed as a degree of freedom. 


We model the multi-vehicle routing problem under energy, time window, and covering constraints. We formulate a deterministic mixed\--integer second-order cone programming (MISOCP) model to construct a set of maximal coverage routes that service (observe) a secondary set of sites under a fixed time schedule. Our model is an extension of both the MCTP and the VRPTW under energy constraints. Our model's solution accounts for the vehicles' energy limitations and prioritizes risk\--averse routes. We implement our model by examining the surveillance of both low and high priority targets. The remainder of the paper is organized as follows: We define our model and solution approach in the Section \ref{sec:dmod} and study a special case which we solve using Lagrangian relaxation in Section \ref{sec:langrangian}. Next, we discuss computational experiments in Section \ref{sec:numexperiments} and conclude with discussions on future research directions in Section \ref{sec:conc}.

\section{Model}\label{sec:dmod}
We seek to route unmanned vehicles for the purpose of target surveillance. In particular, we consider a fleet of vehicles that are used to cover a set of stationary targets. These targets, moreover, can be covered by traveling in proximity to and by idling near them. The process of surveilling targets of interest immediately subjects the fleet to threats of detection. Supplementary problems thus aim to evade enemy threats such as radar detection or combatant encounters (see, e.g., \citealp{Alotaibi2018, Xia2017, Zabarankin2006}), and hence seek to minimize the potential risk associated with a vehicle's route. We study the tradeoff between developing routes that are both risk-averse and coverage\--intensive, while accounting for the limitations imposed on the fleet with respect to time, energy, speed, and risk.

Specifically, the model determines (i) the number of vehicles to deploy and their respective routes, (ii) the vehicle speed along every traversed arc, and (iii) the service time, i.e., the time spent surveilling while idling at a given waypoint, such that the minimum level of surveillance required for each target is achieved. Vehicle energy limits are considered and high-valued surveillance targets are prioritized. Coverage radii are defined per vehicle such that a target must pass within the coverage area for some specified time to successfully surveil the respective target. Moreover, risk radii are defined per target such that a vehicle's threat of detection is dependent upon the amount of time spent within some range of a given target.

\subsection{Notation and Model Formulation}
\label{sec:model}
We formulate the proposed model over a directed network $G=(V\cup W,A)$ with a set of $K$ vehicles. The node set $V = \{0,1,\ldots, n,n+1\}$ consists of all waypoints (locations) that can be visited, where $0$ and ${n+1}$ are the insertion and extraction depots, respectively. We assume that waypoints have been previously determined and are fixed. For notational purposes, we define $V^{0}=V\backslash\{0,{n+1}\}$. 
Let $W$ be the set of targets and $A$ be the set of arcs in the network. 

Routing decisions are incorporated for each arc/vehicle pair via a binary decision variable $x_{ijk}$ which equals $1$ if vehicle $k$ travels along arc $(i,j)\in A$ and $0$ otherwise. In the context of surveillance, vehicles may only visit particular waypoints during any specific period of availability. This could be due to security rounds or temporary site clearances.
Hence, a time window $[a_i,b_i]$ is associated with each waypoint $i\in V$ for a given time horizon, where $[a_{0},b_{0}] = [a_{{n+1}},b_{{n+1}}]$ represents the earliest departure and the latest arrival times to the insertion and extraction depots, respectively. 
The resulting set of routes can be repeated in each time interval to account for a periodic time horizon. The decision variable $s_{ik}$ represents the time at which vehicle $k$ arrives at waypoint $i\in V$, whereas decision variable $y_{ik}$ represents the amount of time vehicle $k$ idles at location $i\in V^{0}$. We do not allow surveillance to be conducted at the depots. Idling (or servicing) and waiting are defined as loitering while active and inactive, respectively. We assume that the time a vehicle spends waiting at a waypoint prior to conducting surveillance does not influence the vehicle's risk of detection nor its level of observation, and vehicles can only conduct service at a waypoint $i\in V^0$ within its associated time window $[a_i,b_i]$. Therefore, vehicles can arrive prior to the start of the time window and wait to begin service. Moreover, let $\mathbf{d}({i,j})$ define the distance between nodes $i\in V$ and $j \in V\cup W$. We assume that each vehicle $k$ travels at a constant speed $v_{ijk}$ on each arc $(i,j)\in A$, so that $\mathbf{d}({i,j}) = t_{ijk}\times v_{ijk}$, where $t_{ijk}$ is the travel time for vehicle $k\in K$ to traverse arc $(i,j)\in A$.

\subsubsection{Risk}\label{sec:risk}
We assume that each target $w\in W$ has an associated risk radius, $\hat{\eta}_w$, which defines a circular area around it in which a vehicle can be detected (e.g., by a radar detection system installed at the target). Due to a vehicle's ability to travel with different speeds on different arcs, we define the risk in terms of a risk index per unit time. The risk index is directly related to the risk factor, $\sigma_w$, and indirectly related to the Euclidean distance between the target and the vehicle (we assume a horizontal plane of motion). The former depends upon the target's detection attributes and is assumed to be a constant. Let $[i,j]$ denote the line segment defined by arc $(i,j)\in A$. We define $[i^w,j^w]\subseteq[i,j]$ as the intersection of line segment $[i,j]$ with the detection area associated with a target $w$ (see Figure \ref{fig: riskarea} for an illustration).
\begin{figure}[htbp]
\centering
\begin{subfigure}[t]{.5\textwidth}
  \centering
  \includegraphics[width=0.75\textwidth]{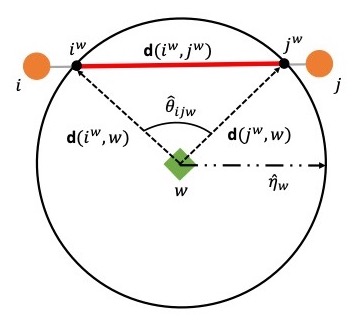}
  \captionsetup{justification=centering}
\caption[caption]{Detection area associated with a target $w$}
\label{fig: riskarea}
\end{subfigure}%
\begin{subfigure}[t]{.5\textwidth}
  \centering
  \includegraphics[width=0.76\textwidth]{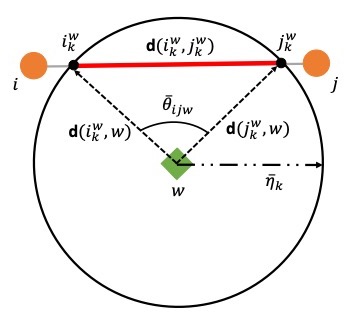}
\captionsetup{justification=centering}
\caption{Coverage area associated with a vehicle $k$\hspace{\textwidth} surrounding a target $w$}
\label{fig: coveragearea}
\end{subfigure}
\caption{Illustration of detection and coverage areas}
\end{figure}

Thus, as in \citet{Zabarankin2002}, the associated risk index for a vehicle at location $h\in[i^w,j^w]$ with respect to target $w\in W$ per unit time is defined as
\begin{equation}
r_{hw} = \sigma_w\frac{1}{\mathbf{d}({h,w})^2}, 
\end{equation}
where $\mathbf{d}(\cdot,\cdot)$ gives the Euclidean distance between two points (for notational convenience, we use $w$ to denote both the index of a target and its location). Using this definition, we can examine a particular segment $[i^w,j^w]$ and its contribution to the total route's risk. Define $P_k$ as the route for vehicle $k\in K$ consisting of a sequence of arcs $(i,j)$ traversed by the vehicle.
A vehicle's total risk accumulated from a specific target is: 
\begin{equation}
\sum_{(i,j)\in P_k}\int_{T([i,j])} r_{h(t)w} dt + \sum_{i\in V^0}r_{iw}y_{ik},
\label{eq:riskint}
\end{equation}
where $y_{ik}$ denotes the amount of time vehicle $k$ idles at waypoint $i\in V^0$, $T(\cdot)$ is the interval of time traveling along a line segment $[i,j]$, $h(t)$ is the vehicle's location at time $t$, and $dt$ is unit time.
Each arc $(i,j)\in P_k$ accrues risk from target $w$ that is proportional to the segment $[i^w,j^w]$ within the detection area. As defined in \citet{Zabarankin2002}, the total risk associated with arc $(i,j)\in P_k\subseteq A$ with respect to target $w\in W$ is:
\begin{equation}\label{eq:riskpart}
\int\displaylimits_{T\([i^w,j^w]\)} r_{h(t)w} dt = \int\displaylimits_{T\([i^w,j^w]\)}\sigma_w\frac{1}{\mathbf{d}({h(t),w})^2} dt = \sigma_w\frac{\hat{\theta}_{ijw}}{\sin{\hat{\theta}_{ijw}}\mathbf{d}({i^w,w})\mathbf{d}({j^w,w})}\frac{\mathbf{d}(i^w,j^w)}{\mathbf{d}(i,j)}t_{ijk},
\end{equation}
where $\hat{\theta}_{ijw}$ represents the angle between two vectors, from $w$ to $i^w$ and from $w$ to $j^w$, as illustrated in Figure~\ref{fig: riskarea}.


\noindent Define for each $(i,j)\in A$:
\begin{equation}
r_{ijw} := \sigma_w\frac{\hat{\theta}_{ijw}}{\sin{\hat{\theta}_{ijw}}\mathbf{d}({i^w,w})\mathbf{d}({j^w,w})},
\end{equation}
and for each $i\in V^0$:
\begin{equation}
r_{iw} := \begin{cases}
 \sigma_w\frac{1}{\mathbf{d}({i,w})^{2}} &\ \text{if}\ \mathbf{d}({i,w}) \le \hat{\eta}_w,\\
0 &\ \text{otherwise}.
\end{cases}
\end{equation}
Using these concepts, the risk incurred by a given vehicle $k\in K$ from target $w\in W$ over the course of its trajectory, i.e., \eqref{eq:riskint}, can be rewritten as:
\begin{equation}
\sum_{w\in W}\left(\sum_{(i,j)\in A}r_{ijw}\frac{\mathbf{d}(i^w,j^w)}{\mathbf{d}({i,j})}  t_{ijk} +\sum_{i\in V^{0}}r_{iw}y_{ik}\right). 
\end{equation}

\subsubsection{Coverage}\label{sec:coverage}
We assume that each vehicle $k\in K$ has an associated coverage area defined by its coverage radius, $\bar{\eta}_k$. Ensuring that a target passes within the coverage area during deployment is equivalent to ensuring that a vehicle is routed within some range of the target. We define $[i_k^w,j_k^w]\subseteq[i,j]$ as the intersection of line segment $[i,j]$ with the coverage area associated with a vehicle $k$ surrounding target $w$. Thus, we can define the level of surveillance obtained by a vehicle equivalently as the risk of detection of the stationary targets in relation to the mobile vehicle (see Figure \ref{fig: coveragearea}).
Similar to the risk of detection as defined in Section \ref{sec:risk}, we can express the total level of surveillance on target $w\in W$ obtained by vehicle $k\in K$ while traversing arc $(i,j)\in A$ as 
\begin{equation}
\rho_k\frac{\bar{\theta}^w_{ijk}}{\sin{\bar{\theta}^w_{ijk}}\mathbf{d}({i_k^w,w})\mathbf{d}({j_k^w,w})}\frac{\mathbf{d}(i_k^w,j_k^w)}{\mathbf{d}(i,j)}t_{ijk},
\end{equation}
where $\rho_k$ is the coverage factor associated with vehicle $k\in K$ representing its surveillance capabilities, and $\bar{\theta}^w_{ijk}$ represents the angle between two vectors, from $w$ to $i^w_k$ and from $w$ to $j^w_k$, as illustrated in Figure~\ref{fig: coveragearea}. 

\noindent Define for each $(i,j)\in A$:
\begin{equation}
c_{ijwk} := \rho_k\frac{\bar{\theta}^w_{ijk}}{\sin{\bar{\theta}^w_{ijk}}\mathbf{d}({i_k^w,w})\mathbf{d}({j_k^w,w})},
\end{equation}
and for each $i\in V^0$:
\begin{equation}
c_{iwk} := \begin{cases}
 \rho_k\frac{1}{\mathbf{d}({i,w})^{2}} &\ \text{if}\ \mathbf{d}({i,w}) \le \bar{\eta}_k,\\
0 &\ \text{otherwise}.
\end{cases}
\end{equation}
Using these concepts, the total level of coverage on target $w$ obtained by a given vehicle $k\in K$ throughout its trajectory is calculated by:
\begin{equation}
\sum_{w\in W}\left(\sum_{(i,j)\in A}c_{ijwk}\frac{\mathbf{d}(i_k^w,j_k^w)}{\mathbf{d}({i,j})}  t_{ijk} +\sum_{i\in V^{0}}c_{iwk}y_{ik}\right). 
\end{equation}

In our proposed model, we make the following additional assumptions. First, we assume that a vehicle can cover multiple locations simultaneously as it traverses an arc, i.e., observation times are non-competitive. Second, we allow multiple vehicles to traverse a given arc. Additionally, we assume that a target $w\in W$ can be surveilled via multiple passes but a vehicle $k\in K$ can visit any waypoint no more than once during the planning horizon. The former is due to the fact that surveillance need not be continual, while the latter is to avoid additional risk brought by multiple visits.

\subsubsection{Energy}\label{sec:energy}
Each vehicle has a limited energy capacity $E^{\text{max}}_k$, such as fuel, battery, etc. It is well-known that the energy consumption for a vehicle to travel along an arc $(i,j)\in A$ depends on the speed at which it travels on the arc. For simplicity, we assume that a vehicle's speed must remain constant over a given arc, and we allow for instantaneous change of speed between different arcs. According to~\citet{Ehsani2018}, the engine power output of a vehicle per unit time can be expressed as a function of its speed: 
\begin{align}
P^e_{ijk} & = v_{ijk}\left(m_vgf_r\cos{(\alpha)}+ \frac{1}{2}\rho_aC_DA_fv_{ijk}^2+M_vg\sin{(\alpha)} + m_v\delta\frac{\partial{v_{ijk}}}{\partial{t_{ijk}}}\right),\\
& = v_{ijk}\left(\beta+ \gamma v_{ijk}^2+\xi + \psi\right),\\
&  = v_{ijk}\left(\beta+ \gamma v_{ijk}^2\right),
\label{poweroutput}
\end{align}
where $\beta$ is the rolling resistance, $\gamma v_{ijk}^2$ represents aerodynamic drag, $\xi$ accounts for the power lost due to the slope of a surface, and $\psi$ is the acceleration of the vehicle. 
We let $\xi = \psi = 0$ since we assume a two-dimensional plane of motion and vehicles are at a steady-state velocity.

Let $E_{ijk}$ be the energy consumption of vehicle $k\in K$ over arc $(i,j)\in A$. Since $E_{ijk} = P^e_{ijk}t_{ijk}$, and $\mathbf{d}({i,j}) = v_{ijk}t_{ijk}$, it follows that 
\begin{align}
E_{ijk} & = \mathbf{d}({i,j})\beta + \gamma \mathbf{d}({i,j}) v_{ijk}^2.
\end{align}
Furthermore, we assume that various power losses due to transmission/drivetrain efficiencies, discharge efficiency, inverter efficiency, motor efficiency, and idling are negligible. We also assume that the vehicle uses no energy (i.e., the vehicle is off) when it waits or idles at a given waypoint.

\subsubsection{Formulation}\label{sec:formulation}

A summary of the notation is provided below.

\noindent {\sc Sets}:
\begin{center} 
\begin{tabular}{p{0.5in}p{5.5in}}
$V$   & Set of nodes that \emph{can} be visited $(V = \{0,1,\ldots, n,n+1\})$ \\
$V^{0}$ & $V\backslash\{0,{n+1}\}$\\
$W$   & Set of targets that \emph{must} be covered $(W=\{w_1,\dots, w_m\})$ \\
$K$   & Set of vehicles $(K=\{k_1,\dots, k_l\})$ \\
$A$   & Set of arcs in $G$, where $A=\left(V\backslash\{{n+1}\}\times V\right)\backslash\left(\{(i,i) : i \in V\} \cup \{(0,{n+1})\}\right)$\\
\end{tabular}
\end{center}
\vskip 1pc
\noindent {\sc Decision Variables}:
\begin{center}
\begin{tabular}{p{0.5in}p{5.5in}}
$x_{ijk}$ & Indicates whether or not arc $(i,j)\in A$ is traversed ($x_{ijk} = 1$) or not ($x_{ijk} = 0$) by vehicle $k\in K$\\
$s_{ik}$ & Time vehicle $k\in K$ arrives at $i\in V$\\
$t_{ijk}$ & Time it takes vehicle $k\in K$ to traverse arc $(i,j)\in A$\\
$v_{ijk}$ & Velocity of vehicle $k\in K$ along arc $(i,j)\in A$\\
$y_{ik}$ & Time vehicle $k\in K$ idles at node $i\in V^{0}$\\
\end{tabular}
\end{center}
\vskip 1pc
\noindent {\sc Parameters}:
\begin{center} \vskip -1pc
\begin{longtable}{p{0.5in}p{5.5in}}
${0}, {n+1}$   & Insertion and extraction depot, respectively\\
$\mathbf{d}({i,j})$   & Distance between nodes $i,j\in V\cup W$\\
$\mathbf{d}(i_k^w,j_k^w)$ & Distance along arc $(i,j)\in A$ in which vehicle $k\in K$ covers target $w\in W$\\
$\mathbf{d}({i^w,j^w})$ & Distance along arc $(i,j)\in A$ in which vehicle $k\in K$ risks detection from target $w\in W$\\
$\bar{p}_w$ & Priority factor for target $w\in W$\\
$\hat{p}_k$ & Priority factor for vehicle $k\in K$\\
$r_{ijw}$ & Risk index for arc $(i,j)\in A$ associated with target $w\in W$\\
$r_{iw}$ & Risk index for waypoint $i\in V$ associated with target $w\in W$\\
$c_{ijwk}$ & Coverage index for arc $(i,j)\in A$ associated with target $w\in W$  for vehicle $k\in K$\\
$c_{iwk}$ & Coverage index for waypoint $i\in V$ associated with target $w\in W$ for vehicle $k\in K$\\
$\sigma_w$ & Risk factor for node $w\in W$\\
$\rho_k$ & Coverage factor for vehicle $k\in K$\\
$t_w$ & Minimum allowable time to survey node $w\in W$\\
${\taubar}_k$ & Maximum speed $(\ge0)$ for vehicle $k\in K$ \\
$\taulow_k$ & Minimum speed $(\ge0)$ for vehicle $k\in K$\\
$E^{\text{max}}_k$ & Maximum energy capacity for vehicle $k\in K$\\
$\beta$ & Rolling resistance coefficient\\
$\gamma$ & Aerodynamic drag coefficient\\
$[a_i,b_i]$ & Time window for node $i\in V$\\
$ \Mbar_{ij}$ & Big-M constant where $\Mbar_{ij} = \max\{b_i+\mathbf{d}({i,j})/\munderline{\tau}-a_j, 0\}$ for all $(i,j)\in A$\\
\end{longtable}
\end{center}
\addtocounter{table}{-1}
\noindent We optimize the objectives
\begin{subequations}
\begin{align}
\label{netcoverage}
&  \max\quad \Ccal = \sum_{w\in W}\bar{p}_w\sum_{k\in K}\left(\sum_{(i,j)\in A}c_{ijwk}\frac{\mathbf{d}(i_k^w,j_k^w)}{\mathbf{d}({i,j})}  t_{ijk} +\sum_{i\in V^{0}}c_{iwk}y_{ik}\right), \\
\label{netrisk}
& \min\quad \Rcal = \sum_{w\in W}\sum_{k\in K}\hat{p}_k\left(\sum_{(i,j)\in A}r_{ijw}\frac{\mathbf{d}({i^w,j^w})}{\mathbf{d}({i,j})}  t_{ijk} +\sum_{i\in V^{0}}r_{iw}y_{ik}\right), 
\end{align}
subject to
\begingroup
\allowdisplaybreaks
\begin{align}
\label{nodecoverage}
& \sum_{k\in K}\left(\sum_{(i,j)\in A}c_{ijwk}\frac{\mathbf{d}(i_k^w,j_k^w)}{\mathbf{d}({i,j})}  t_{ijk} +\sum_{i\in V^{0}}c_{iwk}y_{ik}\right) \ge t_w, && \forall\ w\in W,\\
\label{drt}
& x_{ijk}\mathbf{d}({i,j}) \le t_{ijk}v_{ijk}, &&\forall\ (i,j)\in A, k\in K,\\
\label{energy}
& \sum_{(i,j)\in A} \left(\mathbf{d}({i,j})\beta x_{ijk} + \gamma \mathbf{d}({i,j}) v_{ijk}^2\right) \le E^{\text{max}}_k, &&\forall\ k\in K,\\
\label{depotcontinuity}
& x_{ijk} \le \sum_{\substack{j'\in V^0:\\ (0,j')\in A}} x_{0j'k}, &&\forall\ (i,j)\in A, k\in K,\\
\label{base1}
&\sum_{\substack{j\in V:\\(0,j)\in A}} x_{{0}jk} = \sum_{\substack{j\in V:\\(j,{n+1})\in A}}x_{j{(n+1)}k}, &&\forall\ k\in K,\\
\label{canvisit}
& \sum_{\substack{j\in V:\\(i,j)\in A}} x_{ijk} \le 1, && \forall\ i\in V, k \in K,\\
\label{flowbalance}
& \sum_{\substack{j\in V:\\ (i,j)\in A}}x_{ijk} = \sum_{\substack{j\in V:\\ (j,i)\in A}}x_{jik},  &&\forall\ i\in V^{0}, k \in K,\\
\label{speedbdt}
& \frac{\mathbf{d}({i,j})}{{\taubar}_k}x_{ijk} \leq t_{ijk} \leq \frac{\mathbf{d}({i,j})}{\taulow_k}x_{ijk}, &&\forall\ (i,j)\in A, k\in K,\\
\label{speedbdv}
& {\taulow_k}x_{ijk} \le v_{ijk} \le {\taubar}_kx_{ijk}, && \forall\ (i,j)\in A, k\in K,\\
\label{arrivaltime}
& s_{ik} + (t_{ijk}+y_{ik}) - s_{jk} \le (1 - x_{ijk})\Mbar_{ij}, && \forall\ (i,j)\in A, k\in K,\\
\label{timewindowslb}
& a_i\sum_{\substack{j\in V:\\(i,j)\in A}}x_{ijk} \le s_{ik}, && \forall\ i\in V^{0}, k\in K,\\
\label{timewindowsub}
& s_{ik} +y_{ik} \le b_i\sum_{\substack{j\in V:\\(i,j)\in A}}x_{ijk}, && \forall\ i\in V^{0}, k\in K,\\
\label{schedule}
& a_0 \le s_{ik} \le b_{n+1}, && \forall\ i\in \{0,{n+1}\}, k\in K,\\
\label{ydepot}
& y_{ik} = 0, && \forall\ i\in \{0,{n+1}\}, k\in K,\\
\label{ywaypoint}
& y_{ik} \ge 0, && \forall\ i \in V^{0}, k\in K,\\
\label{x}
&x_{ijk}\in\{0,1\}, &&\forall\ (i,j)\in A, k\in K.
\end{align}
\endgroup
\label{fullmodel}%
\end{subequations}

Objective \eqref{netcoverage} maximizes a weighted sum of coverage times for all targets as some targets are of higher priority than others, while objective \eqref{netrisk} minimizes the fleet's risk of detection. Constraint set \eqref{nodecoverage} ensures that the required minimum amount of surveillance is obtained for each target. Constraints \eqref{drt} relate distance to steady-state speed and travel time; the equality relationship $\mathbf{d}({i,j}) = v_{ijk}t_{ijk}$ is relaxed to an inequality without affecting the optimal value of \eqref{fullmodel} to allow for a second-order cone reformulation (see Proposition \ref{rotatedconeproposition}). Constraints \eqref{energy} ensure that each vehicle does not exceed its maximum energy capacity. Constraints \eqref{depotcontinuity} ensure that every utilized vehicle is deployed from the insertion depot $0$, while constraints \eqref{base1} make sure that the same number of vehicles that leave the insertion depot arrive at the extraction site ${n+1}$. For a given vehicle $k\in K$ and waypoint $i\in V$, the constraint set \eqref{canvisit} requires that no waypoint is visited more than once per vehicle. Constraints \eqref{flowbalance} enforce flow balance and effectively require that vehicles may only start or end their route at a depot. Constraints \eqref{speedbdt} and \eqref{speedbdv} enforce the vehicle speed restrictions, while constraints \eqref{arrivaltime} -\eqref{schedule} ensure schedule feasibility with respect to time window constraints. In particular, constraints \eqref{arrivaltime} influence vehicle speeds by limiting the amount of time they can devote toward surveilling any subset of targets. Moreover, although a vehicle can arrive at a given waypoint prior its earliest arrival time, constraints \eqref{timewindowslb} ensure that service (surveillance) starts within the time window. Constraints \eqref{timewindowsub} define $y_{ik}$ as the idling time at waypoint $i\in V$ for each vehicle $k\in K$ and do not allow a vehicle to conduct surveillance at a waypoint beyond the time window; constraints  \eqref{timewindowslb} and  \eqref{timewindowsub} restrict the total idle time to be at most the time window's length. 
Constraints \eqref{schedule} limit total time of each vehicle's route and constraints \eqref{ydepot} do not allow service to be conducted at either the insertion or extraction sites. Finally, constraint sets \eqref{ywaypoint} and \eqref{x} are the non-negativity and integrality restrictions on the decision variables, respectively. We have assumed an identical fleet of vehicles in our experiments. That is, all vehicles have the same surveillance range ($\bar{\eta}$), importance factor ($\phat$), coverage factor ($\rho$), energy capacity ($E^{\text{max}}$), and boundary speeds ($\taubar$ and $\taulow$). 

\begin{proposition}
The second-order cone constraint \eqref{drt} will be active for at least one optimal solution. 
\label{rotatedconeproposition} 
\end{proposition}
\begin{proof}
Assume for contradiction that for any optimal solution $(x^{*}, s^{*}, t^{*},v^{*}, y^{*})$ of \eqref{fullmodel}, there exists some $k\in K$ and some $(i,j)\in A$ such that $x^{*}_{ijk} = 1$ and $t_{ijk}^{*}v_{ijk}^{*} > \mathbf{d}({i,j})x^{*}_{ijk}$. In particular, note that
\begin{equation}
{\taulow_k}x^{*}_{ijk} \le v^{*}_{ijk} \le {\taubar}_kx^{*}_{ijk}, \quad \forall\ (i,j)\in A, k\in K,
 \end{equation}
 and
 \begin{equation}
 \sum_{(i,j)\in A} \mathbf{d}({i,j})\beta x^{*}_{ijk} + \gamma \mathbf{d}({i,j}) {v^{*}_{ijk}}^2 \le E^{\text{max}}_k.
 \end{equation}
 Moreover, let $\vbar_{ijk} = \frac{\mathbf{d}({i,j})}{t^{*}_{ijk}}$ such that $t_{ijk}^{*}\vbar_{ijk} = \mathbf{d}({i,j})$. It follows then that $\vbar_{ijk} < v_{ijk}^{*}$ and so
 \begin{equation}
  \sum_{(i,j)\in A} \mathbf{d}({i,j})\beta x^{*}_{ijk} + \gamma \mathbf{d}({i,j}) {\vbar_{ijk}}^2 \le E^{\text{max}}_k.
 \end{equation}
Further, since 
\begin{equation}
 \frac{\mathbf{d}({i,j})}{\bar{\tau}}x^{*}_{ijk} \leq t_{ijk}^{*} \leq \frac{\mathbf{d}({i,j})}{\munderline{\tau}}x^{*}_{ijk},
 \end{equation}
 we know 
 \begin{equation}
{\taulow_k}x^{*}_{ijk}  \leq \vbar_{ijk} \leq {\taubar}_kx^{*}_{ijk}.
 \end{equation}
Thus, there exists an alternative optimal solution in which $t_{ijk}^{*}\vbar_{ijk} = \mathbf{d}({i,j})x^{*}_{ijk}$ for all $(i,j)\in A$ and for all $k\in K$.
\end{proof}

\subsection{Solution Methodology}
\label{sec:sm}

Model \eqref{fullmodel} is a bi-objective optimization problem where the fleet's ability to gather information is hindered by
the risk accumulated during deployment. Thus, our intention is to derive a \emph{Pareto curve} which depicts the tradeoff between the two competing objectives. The \emph{efficient frontier} characterizes the set of non-dominated solutions and can be generated via several methodologies (see, e.g., \citealp{Ehrgott2005}). We use the $\epsilon$-constraint method \citep{Haimes1971}, which can generate all possible \emph{efficient solutions} in the absence of convexity so long as $\epsilon$ is chosen appropriately \citep{Olag71}. We chose to bound the total risk (objective \eqref{netrisk}) with $\epsilon$-constraints. That is, we solve
\begin{subequations}
\begin{align}
\text{max } & \Ccal, \label{epsilonconnobj}\\
\label{epsilonconn1}
\text{s.t. } & \Rcal \le \epsilon, \\
& \text{constraints~\eqref{nodecoverage}--\eqref{x}},\notag
\end{align}
\label{epsilonmodel}%
\end{subequations}
for a given value of $\epsilon$ that admits a feasible solution under constraints \eqref{nodecoverage}--\eqref{x}. Due to the continuous nature of our objective functions, we aim to approximate the Pareto frontier by solving model \eqref{epsilonmodel} iteratively. Thus, given the total risk associated with the $n^\text{th}$ iteration ($\Rcal^n$) and a pre-defined step\--size $\delta$, we define
\begin{equation} 
\epsilon^{n+1}\leftarrow \Rcal^n + \delta.
\end{equation}
\break We can find a sufficient number of non-dominated points by choosing $\delta$ small enough. However, we cannot guarantee we generate the entire Pareto curve. 

\section{Special Cases}
\label{sec:langrangian}
In this section, we consider several special cases of the proposed model and develop specialized solution approaches for these cases. Consider the following simplifying assumptions: 1) the fleet's energy supply and the time windows associated with each waypoint can be approximated by an operational deadline ($T = b_{{n+1}} - a_{0}$), and 2) vehicles are not subject to risk of detection, i.e., the risk constraint~\eqref{epsilonconn1} is relaxed. Specifically, the above assumptions can be fulfilled by letting $[a_i, b_i] = [0,\infty)$ for all $i \in V^{0}$ and $E^{\text{max}}_k = \infty$ for all  $k\in K$, and introducing constraints:
\begin{equation}
\sum_{(i,j)\in A}t_{ijk} + \sum_{i\in V^{0}}y_{ik} \le T, \quad\forall\ k\in K.
\label{operationaldeadlineconstraint}
\end{equation}
We refer to this new model as model (\textup{\ref{epsilonmodel}$'$}). Without the constraints on the risk of detection~\eqref{epsilonconn1}, model (\textup{\ref{epsilonmodel}$'$}) is loosely coupled in the sense that each vehicle $k\in K$ can be treated independently if constraints \eqref{nodecoverage} are relaxed. This motivates a Lagrangian relaxation based approach to solve model (\textup{\ref{epsilonmodel}$'$}). For each $w\in W$, given a Lagrangian multiplier $\lambda_w$ associated with constraint \eqref{nodecoverage} for target $w$, the Lagrangian relaxation of model (\textup{\ref{epsilonmodel}$'$}) is formulated as:
\begin{subequations}\label{epsilonconn0}
\begin{align}
\text{max } & \Ccal - \sum_{w \in W}\lambda_w\Bigg(\sum_{k \in K} \Big(\sum_{(i,j) \in A} c_{ijwk}\frac{\mathbf{d}(i_k^w,j_k^w)}{\mathbf{d}({i,j})} t_{ijk} + \sum_{i \in V^{0}} c_{iwk}y_{ik}\Big)-t_w\Bigg), \\
\text{s.t. } & \text{constraints~\eqref{drt}--\eqref{x}}, \eqref{operationaldeadlineconstraint}.
\end{align}
\end{subequations}
The Lagrangian relaxation \eqref{epsilonconn0} becomes separable by $k\in K$ for each fixed vector $\boldsymbol{\lambda}$ of Lagrangian multipliers.  We define $\Pcal$ to be the set of all feasible paths from node $0$ to node ${n+1}$. Due to the homogeneity assumption of the fleet, define $\dbar_{ijw} := \mathbf{d}(i^w_k,j^w_k)$, $c_{iw} := c_{iwk}$, $c_{ijw} := c_{ijwk}$, for all $k\in K$. Hence, given a fixed $\boldsymbol{\lambda}\in\R{|W|}$, \eqref{epsilonconn0} can be reformulated as:
\begin{subequations}
\begin{align}
f(\boldsymbol{\lambda}) = \sum_{w\in W}t_w\lambda_w + & \max_{P\in\Pcal} \sum_{w\in W}(\bar{p}_w-\lambda_w)\left[\sum_{(i,j)\in P}c_{ijw}\frac{\dbar_{ijw}}{\mathbf{d}({i,j})}  t_{ij} + \sum_{i\in V^{0}}c_{i w}y_{i} \right],\\
\label{casetime}
\text{s.t.} &  \frac{\mathbf{d}({i,j})}{\bar{\tau}} \leq t_{ij} \leq \frac{\munderline{d}({i,j})}{\munderline{\tau}}, &&\forall\ (i,j)\in P,\\
& y_{i} = 0, &&\forall\ i\in \{0,{n+1}\},\\
\label{casesy}
& y_{i} \ge 0, && \forall\ i \in V^{0}.
\end{align}
\label{relaxation}%
\end{subequations}

The tightest Lagrangian relaxation bound, namely the Lagrangian dual bound, can be found by optimizing over the Lagrangian vector $\boldsymbol{\lambda}\in\R{|W|}$. Optimal solutions to the Lagrangian relaxation~\eqref{relaxation} will produce cutting planes that will be essential to solving the Lagrangian dual problem. We consider a standard level bundle method \citep{Lemarechal1995}, a variant of the cutting plane method for solving the Lagrangian dual problem. Specifically, this method keeps track of a lower bound and an upper bound to the Lagrangian dual value, denoted as $LB$ and $UB$, respectively, and updates their values in an iterative fashion until they converge to the same value, which corresponds to the Lagrangian dual value. Suppose $\boldsymbol{\hat{\lambda}^L}$ is the iterate solution and suppose $f^\text{lev}_L$ is the level parameter at the $L$-th iteration of the cutting plane method, the Lagrangian dual master problem is given by:
\begin{subequations}\label{master}
\begin{align}
\min\quad &\|\boldsymbol{\lambda} -\boldsymbol{\hat{\lambda}^L}\|^2 \\
 \text{s.t.}\quad &\theta \le f^\text{lev}_L,\\
 \label{levelcut}
& \theta  \ge \sum_{w\in W} \left[t_w - \hat{\alpha}_w(\boldsymbol{\hat{\lambda}^\ell})\right]\boldsymbol{\lambda}_w + \sum_{w\in W}\pbar_w\hat{\alpha}_w(\boldsymbol{\hat{\lambda}^\ell}),\quad\quad\forall \ell = 1,2,\ldots, L,\\\
&\boldsymbol{\lambda}\le 0,
\end{align}
\label{levelbundle}%
\end{subequations}
where cut coefficient $\hat{\alpha}_w(\boldsymbol{\hat{\lambda}^\ell})$ is obtained by solving the Lagrangian relaxation \eqref{relaxation} for Lagrangian vectors $\boldsymbol{\hat{\lambda}^\ell}\in\R{|W|}, \ \ell = 1,2,\ldots, L$. Solving this Lagrangian dual master problem will produce a new iterate $\boldsymbol{\hat{\lambda}^{L+1}}$, on top of which a cut of the form \eqref{levelcut} will be generated by solving the corresponding Lagrangian relaxation problem~\eqref{relaxation} with $\boldsymbol{\hat{\lambda}^{L+1}}$, which will then be added to the master problem~\eqref{levelbundle} in the next iteration. The level parameter $f^\text{lev}_L$ is defined as $f^\text{lev}_L: = \phi LB + (1-\phi) UB$ for some $\phi\in (0,1)$. $UB$ is updated each time a Lagrangian relaxation problem~\eqref{relaxation} with respect to a new iterate is solved: $UB \leftarrow \min\{UB,f(\boldsymbol{\lambda^{\ell}})\}$, and $LB$ is updated each time where the Lagrangian master problem~\eqref{master} is infeasible, indicating that $LB$ can be updated to be identical to the level parameter $f^\text{lev}_L$. 

\paragraph{Optimal solution structures of the Lagrangian relaxation problem~\eqref{relaxation}.} We show that an optimal solution to the Lagrangian relaxation problem~\eqref{relaxation}, which will be iteratively solved within the level bundle method for computing the Lagrangian dual value, exhibits certain combinatorial structures that allow us to develop specialized solution approaches. 

\begin{proposition}\label{1wait}
There exists an optimal solution to problem~\eqref{relaxation} where a vehicle will idle at no more than one waypoint, i.e., among all waypoints $i\in V^0$, there exists at most one waypoint $\bar{v}$ where $y_{\bar{v}} > 0$.
\end{proposition}
\begin{proof}
Assume for contradiction that for any optimal solution to~\eqref{relaxation}, there exists a set $I$ of waypoints such that $y_i > 0, i\in I$, and $|I| \geq 2$. Let $\vbar\in I$ be such that $c_{\vbar w}\sum_{w\in W}(\bar{p}_w-\lambda_w) \ge c_{iw}\sum_{w\in W}(\bar{p}_w-\lambda_w)$ for all $i\in I\backslash\{\vbar\}$. The solution could be improved by ``reallocating'' all the idling time to waypoint $\vbar$, while the travel times $t_{ij}$ for all arcs $(i,j)$ traversed by the route corresponding to the solution are fixed. Thus, the assumed solution must not have been optimal, which yields a contradiction. In the case where there exist multiple $\vbar$, an alternative optimal solution can be obtained by idling only at one $\vbar$, which also yields a contradiction.
\end{proof}
\remark{The optimal solution structure outlined in Proposition \ref{1wait} only holds for the Lagrangian relaxation problem~\eqref{relaxation}, but not for the original model (\textup{\ref{epsilonmodel}$'$}). Consider the following example in which there are two targets where each has only one waypoint in range. Further, suppose the waypoints are located at the boundary of the coverage area so that no coverage can be obtained by traversing any arc, and also assume that each target requires a positive amount of coverage (see constraint set \eqref{nodecoverage}). In this case, to meet the coverage level a vehicle must idle at both waypoints.}

Next, we develop a specialized algorithm for solving the Lagrangian relaxation problem~\eqref{relaxation}, which exploits the aforementioned optimal solution structures. We will start with a simpler case where $T\ge \max_{P\in\Pcal}\sum_{(i,j)\in P}\frac{\mathbf{d}({i,j})}{\taulow}$, i.e., the operational deadline is irrelevant as any path is feasible under this time limit even traveling with the slowest possible speed. 

\subsection{Special Case I: A Sufficiently Large Operational Deadline}\label{sec:case1}
In this case, since the operational deadline is sufficiently large, i.e., $T\ge \max_{P\in\Pcal}\sum_{(i,j)\in P}\frac{\mathbf{d}({i,j})}{\taulow}$, given a path $P$ and the travel time on each arc $(i,j)\in P$, if the vehicle ever idles at any waypoint, it will spend $\sum_{i\in V^{0}}y_{i} = T - \sum_{(i,j)\in P} t_{ij}$ on exactly one waypoint along path $P$ according to Proposition~\ref{1wait}. In the following approach, we will loop through all candidate waypoints $\vbar\in V^0$ where a vehicle may benefit from idling (to gain a positive coverage), i.e., $\sum_{w\in W}(\bar{p}_w-\lambda_w) c_{\vbar w} > 0$, and find the best path that traverses this waypoint. We must also consider the case where the vehicle does not idle at all. After we finish the loop, the best path among all these paths is selected.  As an overview, to solve~\eqref{relaxation}, we use the following procedure: (i) For each $\vbar\in I:=\{i\in V^0 \mid \sum_{w\in W}(\bar{p}_w-\lambda_w) c_{i w} > 0\}\cup\{0\}$, find the maximal covering route $P^{*}_{\vbar}$ among all feasible routes $\Pcal_{\vbar}$ that travel through $\vbar$. We denote the case when no idling occurs as $\vbar = 0$ and define $c_{0w} := 0, \forall w\in W$. Arc weights represent the coverage obtained while traveling along each respective route segment, whose values are known upon fixing $\vbar$; 
ii) Select the optimal path (and the associated travel and idle time along the path) from $\cup_{\vbar\in I}\{P^{*}_{\vbar}\}$.

Given a fixed waypoint $\vbar$ where the vehicle idles, model~\eqref{relaxation} becomes:
\begin{subequations}
\begin{align}
\label{case1obj}
f_{\vbar}(\boldsymbol{\lambda}) = \sum_{w\in W}t_w\lambda_w + & \max_{P\in\Pcal_{\vbar}} \sum_{w\in W}(\bar{p}_w-\lambda_w)\left[\sum_{(i,j)\in P}c_{ijw}\frac{\dbar_{ijw}}{\mathbf{d}({i,j})}  t_{ij} + c_{\vbar w}y_{\vbar} \right],\\
\label{case1timelimitb}
 & \text{s.t.} \ \frac{\mathbf{d}({i,j})}{\bar{\tau}} \leq t_{ij} \leq \frac{\mathbf{d}({i,j})}{\munderline{\tau}},  &&\forall\ (i,j)\in P,\\
\label{case1yvbar}
 & y_{\vbar} \ge 0, &&\forall\ \vbar\in V^0.
\end{align}
\label{case1model}%
\end{subequations}
Since $y_{\vbar} = T - \sum_{(i,j)\in P}t_{ij}$, let $P^{*}_{\vbar}$ and $t^{*}_{ij}$ for all $(i,j)\in P^{*}_{\vbar}$ 
denote an optimal solution to model~\eqref{case1model}, we can rewrite its objective function~\eqref{case1obj} as:
\begin{align}
f_{\vbar}(\boldsymbol{\lambda}) = &\sum_{w\in W}t_w\lambda_w +\sum_{w\in W}(\bar{p}_w-\lambda_w)T c_{\vbar w} + \max_{P\in\Pcal_{\vbar}} \sum_{(i,j)\in P}\left[\sum_{w\in W}(\bar{p}_w-\lambda_w)\left(c_{ijw}\frac{\dbar_{ijw}}{\mathbf{d}({i,j})} - c_{\vbar w}\right)t_{ij}\right],\nonumber\\
 & = \sum_{w\in W}t_w\lambda_w +\sum_{w\in W}(\bar{p}_w-\lambda_w)T c_{\vbar w} + \sum_{(i,j)\in P^{*}_{\vbar}}\left[\sum_{w\in W}(\bar{p}_w-\lambda_w)\left(c_{ijw}\frac{\dbar_{ijw}}{\mathbf{d}({i,j})} - c_{\vbar w}\right)t^{*}_{ij\bar{v}} \right],\nonumber\\
& =  \sum_{w\in W}t_w\lambda_w +\sum_{w\in W}(\bar{p}_w-\lambda_w)\underbrace{\left[T c_{\vbar w} + \sum_{(i,j)\in P^{*}_{\vbar}}\left(c_{ijw}\frac{\dbar_{ijw}}{\mathbf{d}({i,j})} - c_{\vbar w}\right)t^{*}_{ij\bar{v}}\right]}_{\hat{\alpha}_w(\boldsymbol{\lambda})},\nonumber\\
& = \sum_{w\in W} \left[t_w - \hat{\alpha}_w(\boldsymbol{\lambda})\right]\lambda_w + \sum_{w\in W}\pbar_w\hat{\alpha}_w(\boldsymbol{\lambda}),
\end{align}
where given $\bar{v}\in I$, the optimal travel time $t^{*}_{ij\bar{v}}$ along each arc $(i,j)\in P^{*}_{\vbar}$ corresponds to either its lower or upper bound. Specifically, for each $(i,j)\in P^{*}_{\vbar}$:
\bitemize
\item[1)] If $\sum_{w\in W}(\bar{p}_w-\lambda_w) \left(\frac{\dbar_{ijw}}{\mathbf{d}({i,j})}   c_{ijw} - c_{\vbar w}\right) > 0$, then $t^*_{ij\bar{v}} = \frac{\mathbf{d}({i,j})}{\taulow}$;
\item[2)] If $\sum_{w\in W}(\bar{p}_w-\lambda_w) \left(\frac{\dbar_{ijw}}{\mathbf{d}({i,j})}   c_{ijw} - c_{\vbar w}\right) \leq 0$, then $t^*_{ij\bar{v}} = \frac{\mathbf{d}({i,j})}{\taubar}$.
\eitemize

Thus, we see that given a fixed waypoint $\vbar$ for which to idle, the speed at which an arc will be traversed immediately
follows, regardless of the choice of path $P_{\vbar}$. Therefore, for each arc $(i,j)\in A$, we can compute the optimal \emph{adjusted coverage} as:
\[
c^*_{(i,j),\bar{v}} := \sum_{w\in W}(\bar{p}_w-\lambda_w) \left(\frac{\dbar_{ijw}}{\mathbf{d}({i,j})}   c_{ijw} - c_{\vbar w}\right)t^*_{ij\bar{v}}.
\]  
This simple observation allows us to perform path/route optimization using a pre-optimized travel time and the associated adjusted coverage for each arc $(i,j)\in A$ as defined above, so that the label correcting algorithm for standard VRP (see, e.g.,~\cite{Feillet2004}) can be applied to solve model~\eqref{case1model} with some simple extensions elaborated below. We remark that this is not the case when the operational deadline $T$ is not sufficiently large.

\subsubsection{The labeling algorithm}\label{sec:lagrangiandual}
Algorithm \ref{labelalg} is an extension of the label correcting algorithm for the VRPTW \citep{Feillet2004}. Specifically, each label $L$ represents a partial path $P$ (i.e., $L=(P)$) starting from $0$ and ending at some node $v\in V^0$. We define $M(L)$ as the set of waypoints visited so far by $L$. For any fixed $\vbar\in I$, we define $c_{\bar{v}}(L)$ as the \emph{cumulative adjusted coverage} associated with $L$, given by $c_{\bar{v}}(L) = \sum_{(i,j)\in P} c^*_{(i,j),\bar{v}}$.

An optimal complete path (starting from node $0$ and ending at node $n+1$) is obtained by performing partial enumeration of labels (e.g., partial paths) via a dynamic programming algorithm that iteratively generates labels by extending existing labels to new waypoints. As the operational deadline is non-binding for any path, a label can be extended to visit any waypoint not already visited.

\begin{algorithm}[htbp]
\caption{A labeling algorithm for solving model~\eqref{case1model} with a given $\bar{v}\in I$.}\label{labelalg}
\begin{algorithmic}[1]
\Statex \underline{\bf Label Initialization} A list of labels $\Lcal(i)$ is stored for each waypoint $i\in V^0$. $\Lcal(i)\leftarrow\emptyset$ and $\delta\leftarrow 1$. Let $c^*_{(i,j),\bar{v}}$ be the optimal adjusted coverage associated with arc $(i,j)\in A$. 
\For {$i\in V^0$}
\State Initialize labels $L^i = ((0,i))$ and $c_{\bar{v}}(L^i) = c^*_{(0,i),\bar{v}}$.  
Add $L^i$ to $\Lcal(i)$.
\EndFor
\Statex \underline{\bf Label Extension}.
\While {$\delta \le |V^0|$}
\For {$j\in V^0$}
\For {$L=(P)$ from $\Lcal(j)$ where $|M(L)|= \delta+1$}
\For {$i\in V^0$}
\If {$i\notin M(L)$}
\parState {Create a label $L^n = (P \bigoplus (j,i))$ (where $P \bigoplus (j,i)$ denotes that arc (j,i) is appended to path $P$) and $c_{\bar{v}}(L^n) =c_{\bar{v}}(L) + c^*_{(j,i),\bar{v}}$.}
\If {$L^n$ is not dominated}
\State {Add $L^n$ to $\Lcal(i)$.}
\For {$L\in \Lcal(i)$}
\State {If $L^n$ dominates $L$, then remove $L$ from $\Lcal(i)$.}
\EndFor
\EndIf
\EndIf
\EndFor
\EndFor
\EndFor
\State {$\delta \leftarrow \delta +1$}
\EndWhile
\Statex \underline{\bf Label Termination}.
\For {$i\in V^0$}
\For {$L=(P)$ from $\Lcal(i)$}
\State Extend label $L$ to $n+1$ and calculate the total cumulative adjusted coverage associated with $L$.
\EndFor
\State {Output label $L$ from $\Lcal(i)$ with the largest cumulative adjusted coverage $c_{\bar{v}}(L)$.}
\EndFor
\end{algorithmic}
\end{algorithm}

An effective dominance rule is critical to the success of the labeling algorithm, which allows many partial paths to be dominated (and thus be removed) from the dynamic programming procedure. We next describe the dominance rule that we develop for special case I.

\begin{proposition}{[Dominance rule for special case I]}\label{case1dominance}
We say that a label $L^2$ is dominated by another label $L^1$, if: (i) $L^1$ and $L^2$ end at the same waypoint; (ii) $M(L^1)\subseteq M(L^2)$; and further:
\begin{itemize}
\item[1.] If $\vbar\in M(L^1)$ and $\vbar\in M(L^2)$ or if $\vbar\notin M(L^1)$ and $\vbar\notin M(L^2)$, then $L^1$ dominates $L^2$ if $c_{\bar{v}}(L^1) \geq c_{\bar{v}}(L^2)$, 
\item[2.] If $\vbar\notin M(L^1)$ but $\vbar\in M(L^2)$, and $\vbar\in V^0$, we compute the worst case extra cost incurred by adjusting a direct application of an extension $L'$ of $L^2$ to $L^1$ in order to accomodate the visit to $\vbar$. This extra cost can be computed as follows:
\begin{equation}\label{vbarextension}
c^{extra}_{\bar{v}}(L^2) = \min_{\substack{i\in V^0:\\ i\notin M(L^2)}}\Big\{c^*_{(i,\bar{v}),\bar{v}} + c^*_{(\bar{v},n+1),\bar{v}} - c^*_{(i,n+1),\bar{v}}\Big\}.
\end{equation}
If $c_{\bar{v}}(L^1)+c^{extra}_{\bar{v}}(L^2) \ge c_{\bar{v}}(L^2)$, then $L^1$ dominates $L^2$.

\end{itemize}
Note that it is not possible for $\vbar\in M(L^1)$ but $\vbar\notin M(L^2)$ since condition (ii) must be satisfied in order for $L^1$ to dominate $L^2$. We define $c^{extra}_{0}(L^2) := 0$.
\end{proposition}

\begin{proof}$ $\newline
\vspace{-1cm}
\begin{itemize}
\item[1.] In this case, any feasible extension $L'$ of $L^2$ to a complete path is also a feasible extension of $L^1$ to a complete path. Also, the cumulative adjusted coverage of the complete path corresponding to $L^1$, $c_{\bar{v}}(L^1\bigoplus L')$, is higher than that of $L^2$, $c_{\bar{v}}(L^2 \bigoplus L')$, since $c_{\bar{v}}(L^1) \geq c_{\bar{v}}(L^2)$ and the extensions for $L^1$ and $L^2$ are identical.

\item[2.] In this case, $L^1$ is not yet a feasible path (since it has not visited $\bar{v}$ yet) and so the argument for case 1 cannot be made. According to conditions (i) and (ii), any feasible extension $L'$ of $L^2$ is also a feasible extension of $L^1$. Suppose that node $i\in V^0, i\notin M(L^2)$ is the last node visited by $L^2$ prior to $n+1$ for this feasible extension $L'$. When we apply the same extension to $L^1$, since $L^1$ has not visited $\bar{v}$ yet, we will have to use arcs $(i,\vbar)$ and $(\vbar,n+1)$ in the extension for $L^1$ instead of arc $(i,n+1)$. As we do not know which node $i\in V^0, i\notin M(L^2)$ is the last node visited by $L^2$ prior to $n+1$, we must consider the worst-case scenario in order to be sufficient for $L^1$ to dominate $L^2$. This worst case corresponds to the expression shown in~\eqref{vbarextension}: if the cumulative adjusted coverage of $L^1$ is greater than or equal to that of $L^2$ in this worst case, then we can conclude that $L^2$ is dominated by $L^1$.
\end{itemize}
\end{proof}

\subsection{Special Case II: A Restrictive Operational Deadline}\label{sec:case2}
We now consider the case where the operational time restriction can limit the path choices. In this case, there exist paths with traversal times potentially longer than the operational time limit, i.e., $T < \max_{P\in\Pcal}\sum_{(i,j)\in P}\frac{\mathbf{d}({i,j})}{\taulow}$. Similar to special case I, we have:
\begin{equation}
f(\boldsymbol{\lambda}) = \sum_{w\in W}t_w\lambda_w +\max 
\resizebox{.75\textwidth}{!}{$\displaystyle
\begin{cases}
\begin{aligned}[c]
\max_{P\in\Pcal} &\sum_{(i,j)\in P}\left[\sum_{w\in W}(\bar{p}_w-\lambda_w)c_{ijw}\frac{\dbar_{ijw}}{\mathbf{d}({i,j})} \right] t_{ij} ,\\
\text{s.t.}&\sum_{(i,j)\in P}t_{ij}  \le T,\\
\phantom{\text{s.t.}}&\frac{\mathbf{d}({i,j})}{\bar{\tau}} \leq t_{ij} \leq \frac{\mathbf{d}({i,j})}{\munderline{\tau}},  \quad\quad\forall\ (i,j)\in P,\\
\end{aligned}&(a)\\
\begin{aligned}[c]
\max_{\vbar\in V^{0}} &\left\{\sum_{w\in W}(\bar{p}_w-\lambda_w)Tc_{\vbar w} +\max_{P\in\Pcal_{\vbar}}\sum_{(i,j)\in P} \left[\sum_{w\in W}(\bar{p}_w-\lambda_w)\left(c_{ijw}\frac{\dbar_{ijw}}{\mathbf{d}({i,j})} - c_{\vbar w}\right)\right]t_{ij}\right\} ,\\
\text{s.t.}&\sum_{(i,j)\in P}t_{ij}  \le T,\\
\phantom{\text{s.t.}}&\frac{\mathbf{d}({i,j})}{\bar{\tau}} \leq t_{ij} \leq \frac{\mathbf{d}({i,j})}{\munderline{\tau}},  \quad\quad\forall\ (i,j)\in P.\\
\end{aligned}&(b)
\end{cases}
$}
\label{case2model}
\end{equation}

In either case (a) or (b), for any fixed path $P\in \Pcal$ or $P\in \Pcal_{\vbar}$, the underlying problem is a \emph{continuous knapsack problem}. Therefore, in an optimal solution, all decision variables will be set to either their lower or upper bounds, save for at most one~\citep{Dantzig1957}. Moreover, the assignment of different values to these variables depends on the ranking of their respective coverage value per unit distance. For notational convenience, we denote $f_{(i,j),\bar{v}} := \sum_{w\in W}(\bar{p}_w-\lambda_w)\left(c_{ijw}\frac{\dbar_{ijw}}{\mathbf{d}({i,j})} - c_{\vbar w}\right)$ for any fixed $\vbar\in I$, for each $(i,j)\in A$.

\begin{proposition}{[Optimal travel time structures on a fixed path]}\label{continuous-knapsack}
Given a fixed $\vbar\in I$, let $P\in \Pcal_{\vbar}$ be a fixed path that traverses $\vbar$, then there is an optimal solution $\{t^*_{ij}\}_{(i,j)\in P}$ to the following optimal timing problem:
\begin{subequations}
\begin{align}
\max \ & \sum_{(i,j)\in P} f_{(i,j),\bar{v}}t_{ij} \\
\text{s.t.} & \sum_{(i,j)\in P}t_{ij}  \le T \\
 & \frac{\mathbf{d}({i,j})}{\bar{\tau}} \leq t_{ij} \leq \frac{\mathbf{d}({i,j})}{\munderline{\tau}},  \quad\quad\forall\ (i,j)\in P,
\end{align}
\end{subequations}
such that $\exists a^*\in P$ where $t^*_{a^*} \in [\frac{\mathbf{d}(a^*)}{\bar{\tau}},\frac{\mathbf{d}(a^*)}{\munderline{\tau}}]$, and:
\begin{itemize}
\item For all $(i,j)\in P, (i,j)\neq a^*$ such that $\frac{f_{(i,j),\bar{v}}}{\mathbf{d}(i,j)} \geq \frac{f_{a^*,\bar{v}}}{\mathbf{d}(a^*)}$, $t^*_{ij} = \frac{\mathbf{d}({i,j})}{\munderline{\tau}}$.
\item For all $(i,j)\in P, (i,j)\neq a^*$ such that $\frac{f_{(i,j),\bar{v}}}{\mathbf{d}(i,j)} \leq \frac{f_{a^*,\bar{v}}}{\mathbf{d}(a^*)}$, $t^*_{ij} = \frac{\mathbf{d}({i,j})}{\bar{\tau}}$.
\end{itemize}
\end{proposition}

We will exploit special optimal solution structures shown in Proposition~\ref{continuous-knapsack} to develop a novel labeling algorithm to solve~\eqref{case2model}. In particular, we introduce the concept of a \emph{token}, which can be used to \emph{allow} the time spent on traversing an arc $(i,j)\in P$ to be strictly between their respective upper and lower bounds. According to Proposition~\ref{continuous-knapsack}, for each route we are allowed to use at most one token, and except for the arc where the token is applied, all other arcs will be traversed with either the minimum or the maximum possible travel time. Of course, the challenge in solving the optimal timing problem together with optimal path selection as in~\eqref{case2model} is that, before a route is completed, it is unclear whether or not the token will be applied to any arc, and if so, where the token will be applied. However, as indicated by Proposition~\ref{continuous-knapsack}, the place where the token will be applied is restricted by arcs that have been traversed by a label (partial path). Specifically, given a label $L$ and an arc $(i,j)$ on which we perform a label extension, we need to consider up to two different possible label extensions from $L = (P)$ to $(P \bigoplus (i,j))$:
\bitemize
\item[1.] If $f_{(i,j),\bar{v}} \leq 0$, then we must extend the label by traversing arc $(i,j)$ with the minimum time. 
\item[2.] If $f_{(i,j),\bar{v}} > 0$, i.e., traversing arc $(i,j)$ with the maximum time is optimal, then we may extend the label by traversing arc $(i,j)$ with either the minimum time or the maximum time. We will refer to these arcs as the \emph{tradeoff} arcs.
\eitemize
We denote the set of tradeoff arcs associated with a label $L$ as the \emph{tradeoff set} $\Scal(L)$. It is clear that we will only use a token on a tradeoff arc. For each tradeoff arc in $\Scal(L)$, we compute its corresponding ratio $\frac{f_{(i,j),\bar{v}}}{\mathbf{d}(i,j)}$. Given a label (partial path) $L$, we keep track of the minimum ratio $f_{\min,\vbar}(L)$ among all tradeoff arcs within $\Scal(L)$ that have been traveled with the maximum time, and the maximum ratio $f_{\max,\vbar}(L)$ among all tradeoff arcs within $\Scal(L)$ that have been traveled with the minimum time (see Figure \ref{fig:ratio}).

\begin{figure}[htbp]
\centering
 \includegraphics[width=0.35\textwidth]{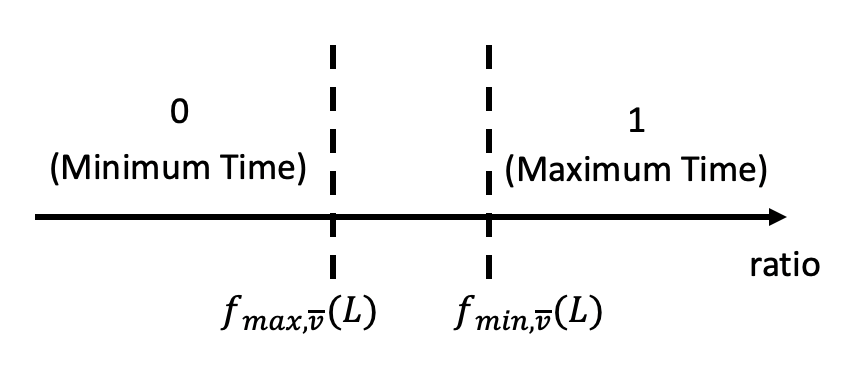}
\caption{Ordering of ratios in tradeoff set}
 \label{fig:ratio}
\end{figure}
Specifically, we label a tradeoff arc $(i,j)\in \Scal(L)$ with label $u_{(i,j)} = 1$ if its travel time is set to its maximum possible travel time, and $u_{(i,j)} = 0$ if its travel time is set to its minimum. When a label $L$ is to be extended with a tradeoff arc $(i,j)$, we first compute its corresponding ratio $\frac{f_{(i,j),\bar{v}}}{\mathbf{d}(i,j)}$, and then we compare it with the minimum and maximum ratio that we keep track of for label $L$ to decide which label ($0$ or $1$) will be assigned to arc $(i,j)$.
\begin{itemize}
\item If $\frac{f_{(i,j),\bar{v}}}{\mathbf{d}(i,j)} \ge f_{\min,\vbar}(L) := \min\left\{\frac{f_{(i',j'),\bar{v}}}{\mathbf{d}(i',j')} \mid (i',j') \in \Scal(L)\ \text{such that}\ u_{(i',j')} = 1\right\}$, then we will label arc $(i,j)$ with $u_{(i,j)} = 1$ and extend label $L$ with maximum travel time on arc $(i,j)$.

\item If $\frac{f_{(i,j),\bar{v}}}{\mathbf{d}(i,j)} \le f_{\max,\vbar}(L) := \max\left\{\frac{f_{(i',j'),\bar{v}}}{\mathbf{d}(i',j')} \mid (i',j') \in \Scal(L)\ \text{such that}\ u_{(i',j')} = 0\right\}$, then we will label arc $(i,j)$ with $u_{(i,j)} = 0$ and extend label $L$ with minimum travel time on arc $(i,j)$.

\item If neither of the above two cases occurs, we will create two new labels where one is extended with the maximum travel time and the other is extended with the minimum travel time. In this case, $f_{\min,\vbar}(L)$ and $f_{\max,\vbar}(L)$ after the extension will be updated accordingly.
\end{itemize}

Note that the feasibility of each label extension is verified before we execute the extension. For any fixed $\vbar\in I$, let $t_{\vbar}(L)$ be the arrival time at the last node visited by $L$, given by 
\[
t_{\bar{v}}(L) = \sum_{(i,j)\in P, (i,j)\notin \Scal(L)} \frac{\mathbf{d}({i,j})}{\bar{\tau}} + \sum_{(i,j)\in P, (i,j)\in \Scal(L)}\left(\frac{\mathbf{d}({i,j})}{\bar{\tau}}(1-u_{(i,j)}) + \frac{\mathbf{d}({i,j})}{\munderline{\tau}}u_{(i,j)}\right). 
\]
We will check:
\begin{itemize}
\item If $\vbar\in M(L)$, a label $L$ ending at waypoint $i\in V^0$ can be extended to $j\in V^0$, where $j\notin M(L)$, only if $t_{\vbar}(L) + \frac{\mathbf{d}({i,j})}{\bar{\tau}}(1-u_{(i,j)}) + \frac{\mathbf{d}({i,j})}{\munderline{\tau}}u_{(i,j)} + \frac{\mathbf{d}(j,n+1)}{\bar{\tau}} \le T$, i.e., a feasible complete path is possible if using the maximum speed on the last leg of the path, $(j, n+1)$.  
\item If $\vbar\notin M(L)$, a label $L$ ending at waypoint $i\in V^0$ can be extended to $j\in V^0$, where $j\notin M(L)$, only if $t_{\vbar}(L) + \frac{\mathbf{d}({i,j})}{\bar{\tau}}(1-u_{(i,j)}) + \frac{\mathbf{d}({i,j})}{\munderline{\tau}}u_{(i,j)} + \frac{\mathbf{d}(j,\bar{v})}{\bar{\tau}} + \frac{\mathbf{d}(\bar{v},n+1)}{\bar{\tau}} \le T$, i.e., a feasible complete path is possible if using the maximum speed on the last two legs of the path, $(j, \bar{v})$ and $(\bar{v}, n+1)$.
\end{itemize}

Except for the possibility of creating up to two new labels per label extension, the labeling algorithm for special case II is identical to Algorithm~\ref{labelalg}. Together with the cumulative adjusted coverage $c_{\vbar}(L)$, the arrival time $t_{\vbar}(L)$, we also store the maximum cumulative adjusted coverage $c^{\max}_{\vbar}(L)$, the minimum cumulative adjusted coverage $c^{\min}_{\vbar}(L)$, and their corresponding arrival times $t^{\max}_{\vbar}(L)$ and $t^{\min}_{\vbar}(L)$, respectively, by applying the token on either the arc corresponding to $f_{\min,\vbar}(L)$ or the arc corresponding to $f_{\max,\vbar}(L)$. We next present the dominance rules that we develop for the labeling algorithm for special case II.

\begin{proposition}{[Dominance Rule for Special Case II]}\label{case2dominance}
We say that a label $L^2$ is dominated by another label $L^1$, if: (i) $L^1$ and $L^2$ end at the same waypoint; (ii) $M(L^1)\subseteq M(L^2)$; (iii) $c_{\vbar}(L^1) \ge c_{\vbar}(L^2)$, $t_{\vbar}(L^1) \le t_{\vbar}(L^2)$, and at least one of these holds strictly, and furthermore, one of the following four cases will hold:
\begin{itemize}
\item Case 1: $c_{\bar{v}}(L^1) \geq c^{\max}_{\bar{v}}(L^2)$, and $t_{\bar{v}}(L^1) \leq t^{\min}_{\bar{v}}(L^2)$;
\item Case 2: If $c_{\bar{v}}(L^1) \leq c^{\max}_{\bar{v}}(L^2)$, and $t_{\bar{v}}(L^1) \leq t_{\bar{v}}^{\min}(L^2)$, then $t_{\bar{v}}(L^1) + \left(c_{\bar{v}}^{\max}(L^2)-c_{\bar{v}}(L^1)\right)/f_{\min,\bar{v}}(L) \leq t^{\max}_{\bar{v}}(L^2)$ must hold;
\item Case 3: If $c_{\bar{v}}(L^1) \geq c^{\max}_{\bar{v}}(L^2)$, and $t_{\bar{v}}(L^1) \geq t^{\min}_{\bar{v}}(L^2)$, then $c_{\bar{v}}(L^1) - f_{\max,\bar{v}}(L)\cdot \left(t_{\bar{v}}(L^1) - t^{\min}_{\bar{v}}(L^2)\right) \geq c^{\min}_{\bar{v}}(L^2)$ must hold;
\item Case 4: If $c_{\bar{v}}(L^1) \leq c^{\max}_{\bar{v}}(L^2)$, and $t_{\bar{v}}(L^1) \geq t^{\min}_{\bar{v}}(L^2)$,  then both $t_{\bar{v}}(L^1) + \left(c^{\max}_{\bar{v}}(L^2)-c_{\bar{v}}(L^1)\right)/f_{\min,\bar{v}}(L) \leq t^{\max}_{\bar{v}}(L^2)$ and $c_{\bar{v}}(L^1) - f_{\max,\bar{v}}(L)\cdot \left(t_{\bar{v}}(L^1) - t^{\min}_{\bar{v}}(L^2)\right) \geq c^{\min}_{\bar{v}}(L^2)$ must hold.
\end{itemize}

\end{proposition}

\begin{figure}[ht]
    \centering
    \begin{subfigure}[t]{0.5\textwidth}
        \centering
          \captionsetup{justification=centering}
        \includegraphics[width=0.75\textwidth]{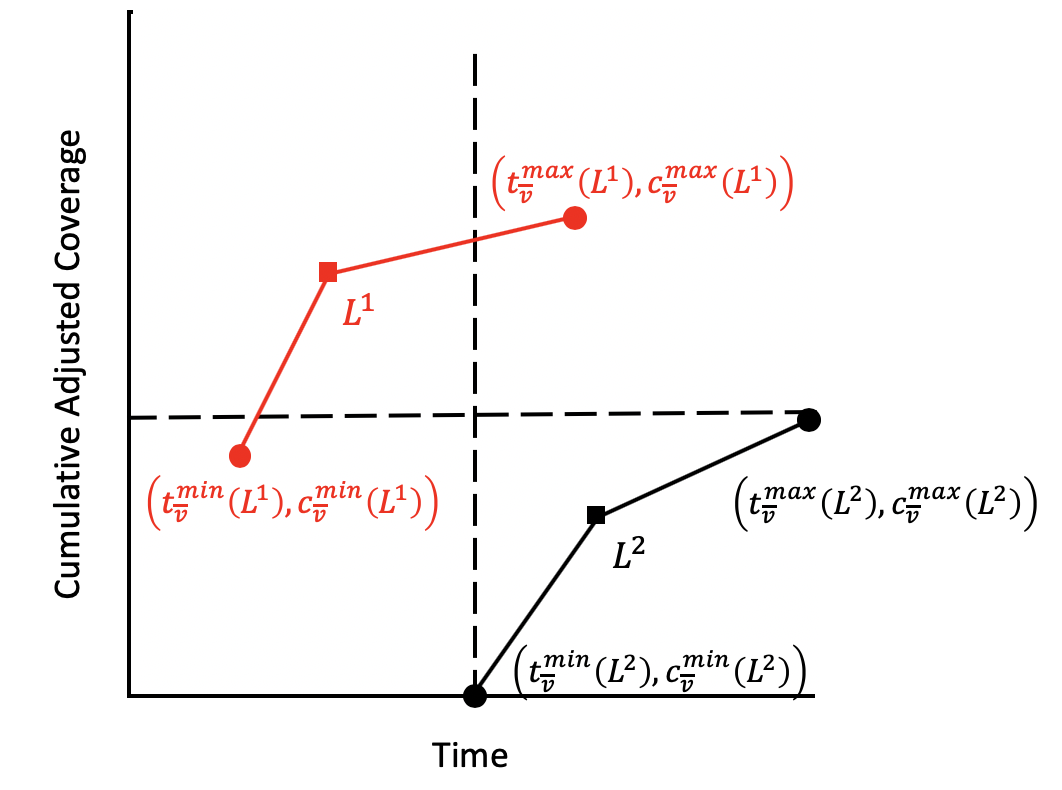}
        \caption[caption]{Case 1}
       \label{fig:trivial}
    \end{subfigure}%
    ~ 
    \begin{subfigure}[t]{0.5\textwidth}
        \centering
          \captionsetup{justification=centering}
        \includegraphics[width=0.75\textwidth]{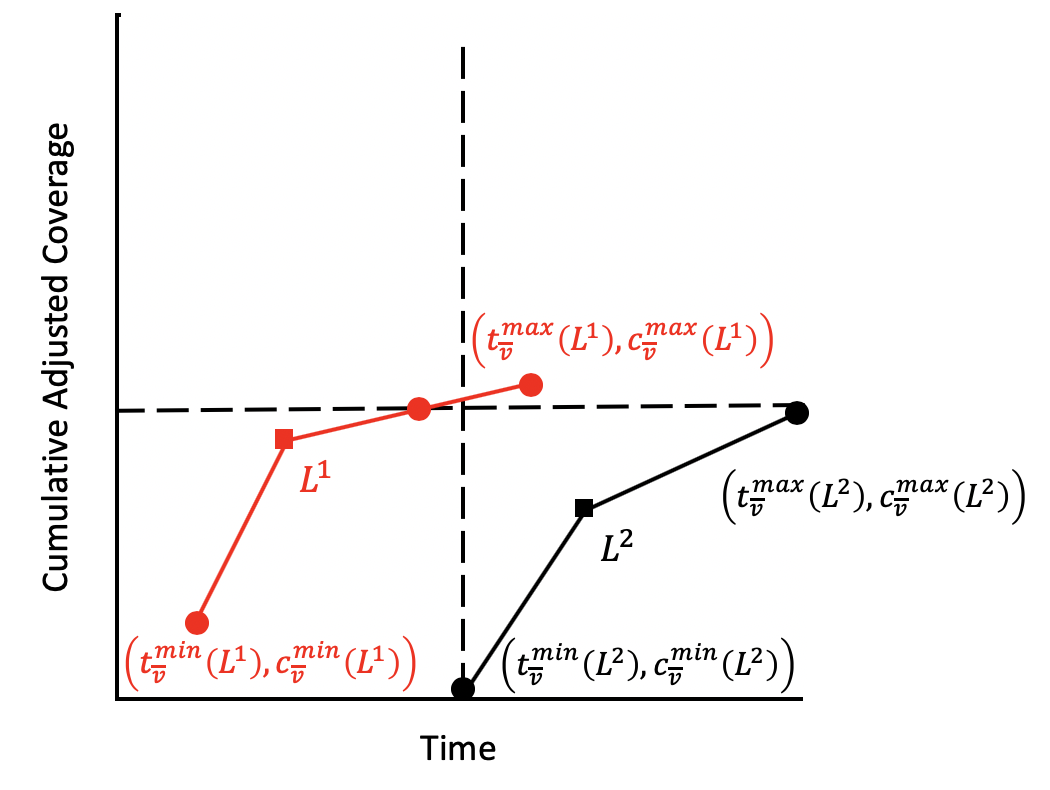}
        \caption[caption]{Case 2}
       \label{fig:reducedcost}
    \end{subfigure}

    \begin{subfigure}[t]{0.5\textwidth}
        \centering
          \captionsetup{justification=centering}
        \includegraphics[width=0.75\textwidth]{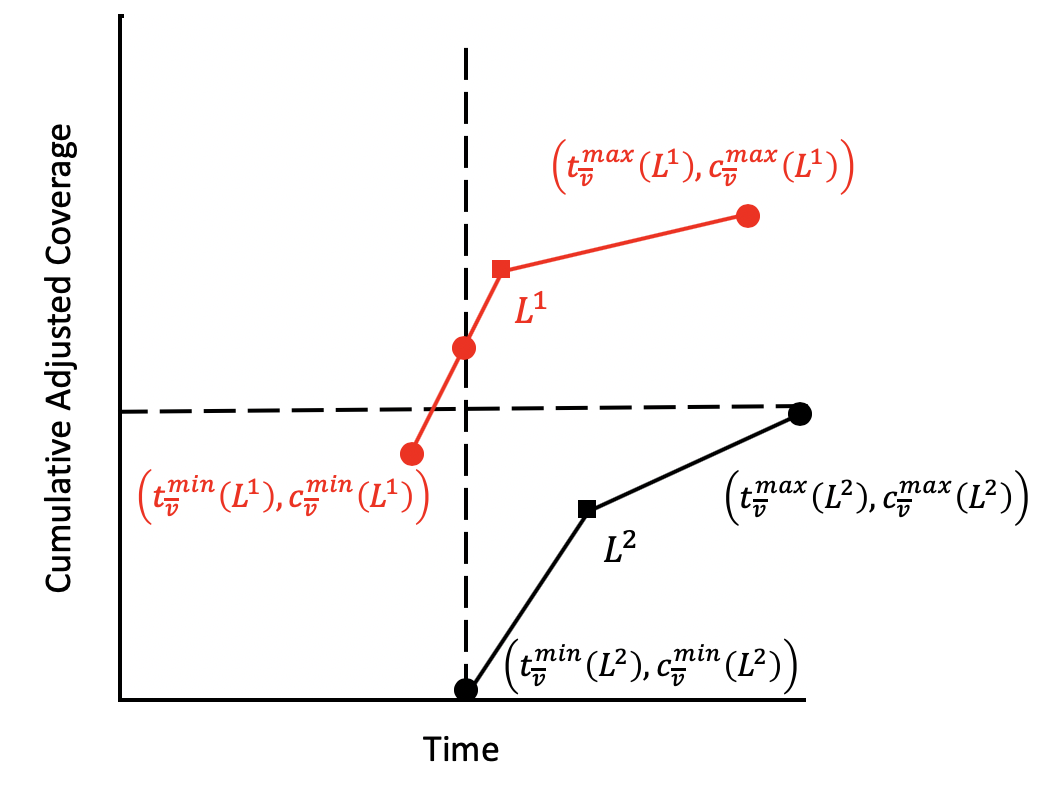}
        \caption[caption]{Case 3}
        \label{fig:time}
    \end{subfigure}%
    ~ 
    \begin{subfigure}[t]{0.5\textwidth}
        \centering
          \captionsetup{justification=centering}
        \includegraphics[width=0.75\textwidth]{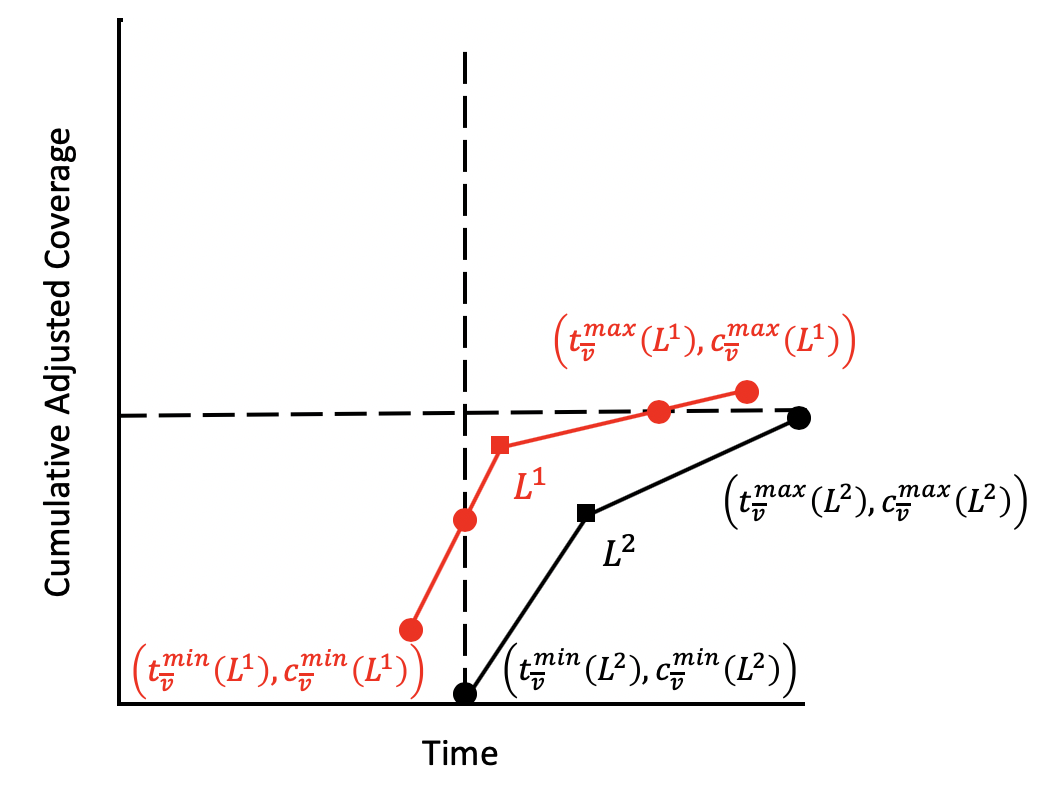}
        \caption[caption]{Case 4}
         \label{fig:both}
    \end{subfigure}

           \caption{Illustrations of the four cases pertaining to the dominance rules presented in Proposition \ref{case2dominance}}
                        \label{fig:dominancerule}
   \end{figure}
   
Figure \ref{fig:dominancerule} illustrates these four cases, where each label $L$ corresponds to a piecewise linear function with two pieces. The upper segment of each function corresponds to using the token to change the traversal time of the arc corresponding to $f_{\max,\vbar}(L)$. As its time can vary between its upper and lower bounds, and in doing so $t_{\vbar}(L)$ and $c_{\vbar}(L)$ increase, the resulting set of solutions are characterized by a line segment. A similar argument is applied to the lower segment associated with changing the traversal time of the arc corresponding to $f_{\min,\vbar}(L)$. We remark that the slope of the lower piece is steeper than that of the upper piece due to the optimal solution structure shown in Proposition~\ref{continuous-knapsack}.

\begin{proof} 
Consider any feasible extension $L'$ of $L^2$. We discuss the following two situations.

First, if the token is used in the extension $L'$ after $L^2$ is extended to a complete path, it is feasible to apply the same extension and use the token at the same arc for $L^1$, and the resulting cumulative adjusted coverage for $L^1$ is higher than $L^2$ according to conditions (i), (ii) and (iii).

Second, if the token is used on an existing arc traversed by label $L^2$, in order to be feasible for $L^1$ to apply the same extension $L'$ of $L^2$, and yet yield a higher cumulative adjusted coverage, we consider the four cases discussed in Proposition~\ref{case2dominance}. Case 1 (Fig.  \ref{fig:trivial}) is clearly trivial, since point $(t_{\vbar}(L^1),c_{\vbar}(L^1))$ dominates both extreme situations for $L^2$ by utilizing a token, that is, $c_{\vbar}(L^1) \geq c^{\max}_{\vbar}(L^2)$, and $t_{\vbar}(L^1) \leq t^{\min}_{\vbar}(L^2)$. For simplicity, we only discuss case 2 in detail below, since case 3 can be shown using a similar argument, and case 4 is a combination of case 2 and case 3.

In fact, conditions imposed in case 2 ensures that the intersection of the upper piece of $L^1$ with $c = c^{\max}_{\vbar}(L^2)$ corresponds to a time that is less than $t^{\max}_{\vbar}(L^2)$. Therefore, both ends of the upper piece of $L^2$ are dominated by some point at the upper piece of $L^1$ (in terms of both cumulative adjusted coverage and time), indicating that for any $(c'_{\vbar}(L^2),t'_{\vbar}(L^2))$ combination associated with $L^2$ that is achievable by applying a token on $L^2$, there exists a combination $(c'_{\vbar}(L^1),t'_{\vbar}(L^1))$ associated with $L^1$ by applying a token on $L^1$ such that $c'_{\vbar}(L^1) \geq c'_{\vbar}(L^2)$ and $t'_{\vbar}(L^1) \leq t'_{\vbar}(L^2)$.
\end{proof}

\section{Numerical Experiments}\label{sec:numexperiments}

We now demonstrate our model using an illustrative example (Section \ref{sec:example}) and further investigate its efficacy with larger problem instances (Section \ref{sec:results}). Next, Section \ref{sec:compresults} examines solution approaches for the special cases discussed in Section \ref{sec:langrangian}. In our experiments, we set $\beta = 1$ and $\gamma = 1$ for the rolling resistance and aerodynamic drag coefficients, respectively, and fix each vehicle's speed range such that $[\taulow,\taubar] = [1,10]$. The model in Section \ref{sec:example} was coded in AMPL and solved via Gurobi 8.0.1.
In Section \ref{sec:results}, computations were performed using an 8\--thread PC running an Intel i7 5960X 3.7 GHz processor with 128 GB RAM and instances were analyzed using Gurobi 8.0.1 with a 3600 seconds time limit. For Section \ref{sec:compresults} computations, an Intel i7-8700 3.2 GHz processor with 32 GB of RAM was used. Gurobi 8.1.1 was run in a Python 2.7 environment with a time limit of 7200 seconds. Both the level bundle method (for solving the Lagrangian dual) and the labeling algorithms (for solving the Lagrangian relaxation outlined in Section \ref{sec:langrangian}) were implemented in C\nolinebreak[4]\hspace{-.05em}\raisebox{.4ex}{\tiny\bf ++}. We use a 7200 second time limit for the level bundle method.

\subsection{Illustrative Examples}
\label{sec:example}

We 
consider a symmetric network composed of four waypoints (\tikz\draw[YellowOrange,fill=YellowOrange] (0,0) circle (.6ex);) and eight targets (\tikz\draw[ForestGreen,fill=ForestGreen] (0,0) rectangle (1ex,1ex);), where the insertion and extraction depots (\hspace{-3mm}\score{1}{1}) are identical (Figure \ref{fig:example_network}). Coverage radii (\tikz\draw [red,thick,dash pattern={on 7pt off 2pt on 1pt off 3pt}] (0,0) -- (0.85,0);) and risk radii (\tikz\draw [black, thick,densely dotted] (0,0) -- (0.85,0);) are consistent for each vehicle and target, respectively. 
For the instance parameters described in Table \ref{tab:example_parameters}, subscripts are omitted as values are made constant over their respective sets. Initially, we do not permit idling at any waypoint. 
\begin{table}
	\begin{minipage}{0.55\linewidth}
	\centering
		\includegraphics[width=0.65\textwidth]{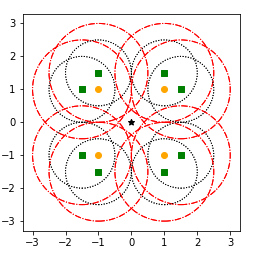}
		\captionof{figure}{Network schematic with $|V^{0}| = 4$ and $|W| = 8$} 
		\label{fig:example_network}
			\end{minipage}\hfill
	\begin{minipage}{0.45\linewidth}
	\caption{Example instance parameters}
		\label{tab:example_parameters}
		\centering
		\begin{tabular}{lc}
\hline
\multicolumn{1}{c}{\textbf{Parameter}} & \textbf{Value} \\ \hline
Number of Vehicles ($|K|$) & 2 \\
Time Windows ($[a,b]$) & $[0,100]$ \\
Energy Capacity $(E^{\text{max}})$ & 25 \\
Coverage Radii $(\bar{\eta})$ & 1.5 \\
Risk Radii $(\bar{\eta})$ & 1 \\
Target $(\pbar)$/Vehicle Priority $(\phat)$ & 1 \\
Risk Threshold $(\epsilon)$ & 50 \\
Minimum Coverage Time $(t)$ & 2 \\
Risk $(\sigma)$ /Coverage $(\rho)$ Factor & 1 \\\hline
\end{tabular}

	\end{minipage}
\end{table}

We perform sensitivity analyses on key parameters: energy capacity, time windows, risk and coverage radii, and risk and coverage levels. We show the impact each parameter has on the solution by varying its value while holding all others constant. It is important to highlight that an optimal solution will utilize the entire fleet. In fact, in the absence of restrictive parameters (e.g., time windows, energy limits, and risk thresholds), model \eqref{fullmodel} will produce identical routes for each vehicle (see Figure \ref{fig:4_8_relaxed}). Decreasing vehicle energy capacity decreases accumulated surveillance since each vehicle must now shorten its route and cover different targets independently (Figure \ref{fig:4_8_energy}). Moreover, by restricting the target's time window for observation, vehicles are assigned different routes to maximize coverage as they must travel more quickly to maintain schedule feasibility (Figure \ref{fig:4_8_timewindow}). When permitted to service waypoints, vehicles can increase the level of observation; however, due to the incurred risk restriction, the fleet cannot idle indefinitely (Figure \ref{fig:4_8_idle}). Decreasing the risk threshold, though, directly reduces the fleet's level of observation (Figure \ref{fig:4_8_threshold}). Increasing the detection area increases the accumulated risk; hence, a vehicle must sacrifice surveillance time in order to satisfy its permissible risk level. By expanding the coverage area, observation increases as a greater proportion of the route covers targets and vehicles travel through detection areas more quickly  (Figure \ref{fig:4_8_coverage}). The speed of travel does not diminish the level of surveillance due to the targets being observable over a greater proportion of the routes. As the model aims to maximize total coverage, decreasing the minimum surveillance level has no effect on the solution, while increasing this requirement can lead to infeasibility (Figure \ref{fig:4_8_level}).

\begin{figure}[ht]
    \centering
    \begin{subfigure}{0.5\textwidth}
        \centering
        \includegraphics[width=\textwidth]{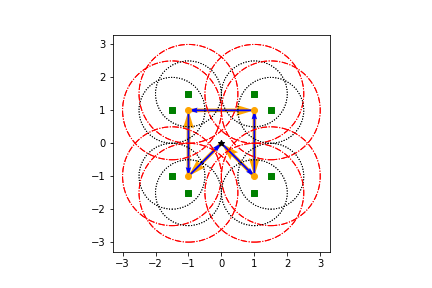}
        \caption{Baseline control parameters ($\Ccal=57.93, \Rcal=46.73)$}
        \label{fig:4_8_relaxed}
    \end{subfigure}%
    ~ 
    \begin{subfigure}{0.5\textwidth}
        \centering
        \includegraphics[width=\textwidth]{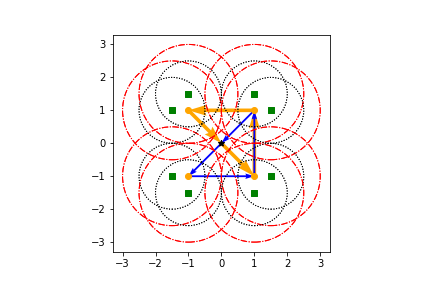}
        \caption{Energy limit $E^\text{max} = 17$ ($\Ccal=42.75, \Rcal=34.35)$}
        \label{fig:4_8_energy}
    \end{subfigure}

        \begin{subfigure}{0.5\textwidth}
        \centering
        \includegraphics[width=\textwidth]{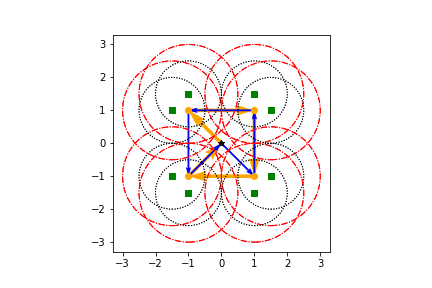}
        \caption{Time window $[a,b] = [0,7]$ ($\Ccal=47.97, \Rcal=38.82)$}
        \label{fig:4_8_timewindow}
    \end{subfigure}%
    ~ 
    \begin{subfigure}{0.5\textwidth}
        \centering
        \includegraphics[width=\textwidth]{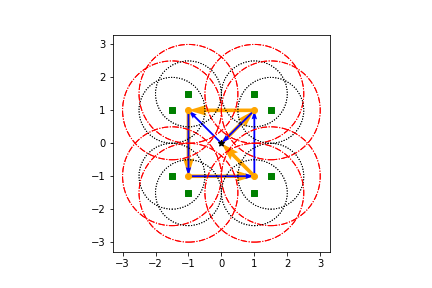}
        \caption{Idling at waypoints ($\Ccal=61.20, \Rcal=50.00)$}
         \label{fig:4_8_idle}
    \end{subfigure}
    \caption{Route schematics when Table \ref{tab:example_parameters} parameters are varied. Parameter changes are denoted in subcaptions}
    \end{figure}
   
\begin{figure}[ht]\ContinuedFloat
            \begin{subfigure}{0.5\textwidth}
        \centering
        \includegraphics[width=\textwidth]{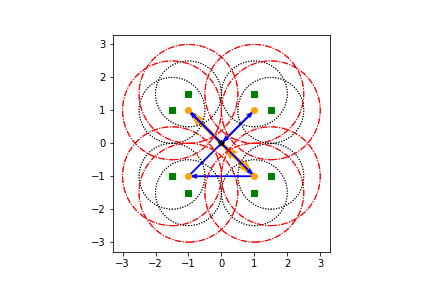}
        \caption{Risk threshold $\epsilon = 25$ $(\Ccal=32.05, \Rcal=25.00$)}
         \label{fig:4_8_threshold}
    \end{subfigure}%
    ~ 
    \begin{subfigure}{0.5\textwidth}
        \centering
        \includegraphics[width=\textwidth]{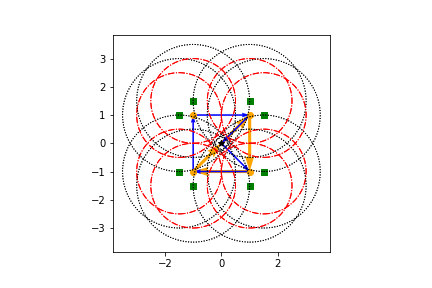}
        \caption{Risk radius $\hat{\eta} = 2$ ($\Ccal=43.79, \Rcal=50.00)$}
         \label{fig:4_8_risk}
    \end{subfigure}
    
     \begin{subfigure}{0.5\textwidth}
        \centering
        \includegraphics[width=\textwidth]{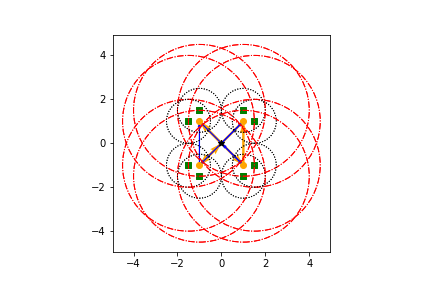}
        \caption{Coverage radius $\bar{\eta} = 3$ ($\Ccal=82.43, \Rcal=41.16)$}
         \label{fig:4_8_coverage}
    \end{subfigure}
     ~ 
                \begin{subfigure}{0.5\textwidth}
        \centering
        \includegraphics[width=\textwidth]{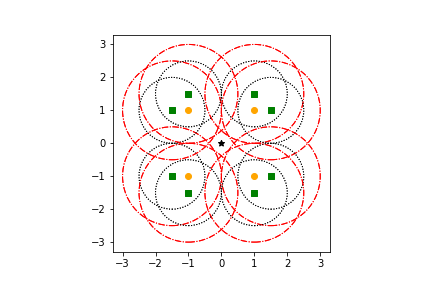}
        \caption{Coverage level $t = 7$ ($\Ccal = -, \Rcal = -$)}
         \label{fig:4_8_level}
    \end{subfigure}%

    \caption{Route schematics when Table \ref{tab:example_parameters} parameters are varied. Parameter changes are denoted in subcaptions}
    \label{fig:4_8fullanalysis}
\end{figure}

As certain restrictions become more or less binding, we not only see routes change, but travel speed as well. Figure \ref{fig:velocity} shows the speed of each vehicle over the duration of the trip subject to baseline control parameters (Figs. \ref{fig:4_8_relaxed1} and \ref{fig:4_8_relaxed2}), decreased energy capacity (Figs. \ref{fig:4_8_energy1} and \ref{fig:4_8_energy2}), shortened time windows (Figs. \ref{fig:4_8_timewindow1} and \ref{fig:4_8_timewindow2}), and permissible service (Figs. \ref{fig:4_8_service1} and \ref{fig:4_8_service2}). The time a vehicle starts service at each waypoint is denoted as a data point, whereas the time spent conducting service at a waypoint is shown as a break in the plot.

\begin{figure}[ht]
    \centering
    \begin{subfigure}{0.49\textwidth}
        \centering
        \includegraphics[width=\textwidth]{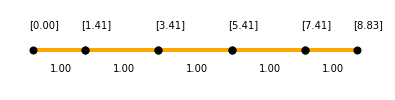}
        \caption{Vehicle 1 velocity (Fig. \ref{fig:4_8_relaxed})}
       \label{fig:4_8_relaxed1}
    \end{subfigure}%
    ~ 
    \begin{subfigure}{0.49\textwidth}
        \centering
        \includegraphics[width=\textwidth]{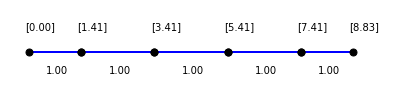}
        \caption{Vehicle 2 velocity (Fig. \ref{fig:4_8_relaxed})}
       \label{fig:4_8_relaxed2}
    \end{subfigure}

    \begin{subfigure}{0.49\textwidth}
        \centering
        \includegraphics[width=\textwidth]{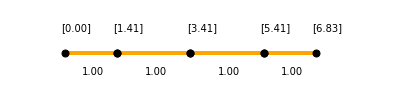}
        \caption{Vehicle 1 velocity (Fig. \ref{fig:4_8_energy})}
        \label{fig:4_8_energy1}
    \end{subfigure}%
    ~ 
    \begin{subfigure}{0.49\textwidth}
        \centering
        \includegraphics[width=\textwidth]{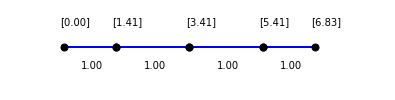}
        \caption{Vehicle 2 velocity (Fig. \ref{fig:4_8_energy})}
         \label{fig:4_8_energy2}
    \end{subfigure}

   

            \begin{subfigure}{0.49\textwidth}
        \centering
        \includegraphics[width=\textwidth]{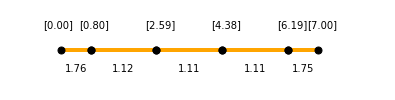}
        \caption{Vehicle 1 velocity (Fig. \ref{fig:4_8_timewindow})}
         \label{fig:4_8_timewindow1}
    \end{subfigure}%
    ~ 
    \begin{subfigure}{0.49\textwidth}
        \centering
        \includegraphics[width=\textwidth]{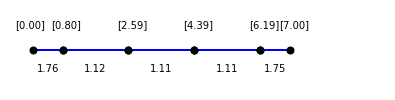}
        \caption{Vehicle 2 velocity (Fig. \ref{fig:4_8_timewindow})}
        \label{fig:4_8_timewindow2}
    \end{subfigure}

    \begin{subfigure}[t]{0.49\textwidth}
        \centering
        \includegraphics[width=\textwidth]{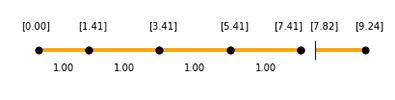}
        \caption{Vehicle 1 velocity (Fig. \ref{fig:4_8_idle})}
          \label{fig:4_8_service1}
    \end{subfigure}\hspace{-2px}
~ 
                \begin{subfigure}[t]{0.49\textwidth}
        \centering
        \includegraphics[width=\textwidth]{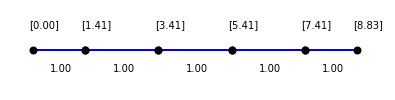}
        \caption{Vehicle 2 velocity (Fig. \ref{fig:4_8_idle})}
          \label{fig:4_8_service2}
    \end{subfigure}%

    \caption{Velocity plots corresponding to routes shown in Figure \ref{fig:4_8fullanalysis}. Arc velocities are denoted below the arc segment. Waypoints are denoted as circular markers with corresponding arrival times in brackets. Vertical lines denote the completion of any idling period}
       \label{fig:velocity}
\end{figure}

Depending upon the limiting resource, the vehicles will either utilize a shorter route, travel at a faster speed, or idle as long as possible. When permissible, vehicles will behave identically, though their paths and speeds will differentiate when resources are limiting or targets are dispersed. In particular, limiting the available energy supply causes a vehicle's total route time to decrease, but the speed it travels remains the same (Figs. \ref{fig:4_8_energy1} and \ref{fig:4_8_energy2}). This decrease in route time creates a corresponding reduction in surveillance. Shortening time windows yields a similar result. However, this causes vehicles to travel faster along certain route segments (Figs. \ref{fig:4_8_timewindow1} and \ref{fig:4_8_timewindow2}). When compared to the effect of decreased energy supply, we see that the increase in speed does not affect the objective as much as the decrease in total trip time; this is due in part to the associated coverage index per time unit. When vehicles are able to service a given waypoint, they will do so for as long as possible such that they can be routed back to the depot within the acceptable risk level (Figs. \ref{fig:4_8_service1} and \ref{fig:4_8_service2}).

\subsection{Larger Problem Instances}\label{sec:results}

Model \eqref{epsilonmodel} was coded in AMPL and validated using instances generated from the well-known Solomon VRPTW benchmarks \citep{Solomon1987}. In particular, we created two distinct networks (Network 1 and Network 2) composed of 20 waypoints (including the depots) and 20 targets from the R101 problem set. The location and time windows for the waypoints are set as the first 20 customers (Customer No. 0\--19) for Network 1 and the second 20 customers (Customer No. 20\--39) for Network 2. Targets are chosen randomly from the remaining 80 customers. The associated time windows for the insertion and extraction depots are redefined for each instance such that $a_0 = 0$ and $b_{n+1} = 10000$. Coverage radii for each vehicle are set to 10 units and the associated risk radius per target is fixed at 5 units. All targets that are not coverable within the given network or lie directly on the arc $(i,j)\in A$ are removed for feasibility and numerical purposes, respectively. 
Table \ref{tab:datasets} shows the common instance parameters for each network. Targets were randomly assigned a priority weight such that $\pbar_w\in[1,5]$ and all vehicles were weighted equally (i.e., $\phat_k = 1,\ \forall k\in K$).

\begin{table}[htbp]
\centering
\caption{Instance parameters for Network 1 and Network 2}
\label{tab:datasets}
\begin{tabular}{l|c}
\hline
\multicolumn{1}{c|}{\textbf{Parameter}} & \textbf{Value} \\ \hline
No. of Waypoints                        & 20             \\
Risk Radius                             & 5              \\
Coverage Radius                         & 10             \\
Coverage Level Per Target               & 1              \\
Energy Capacity                         & 67500          \\
Risk Threshold                          & 500            \\ \hline
\end{tabular}%
\end{table}

It is important to note that Network 1 consists of instances solvable with two vehicles, whereas Network 2 requires a fleet size of at least seven vehicles. The energy limits were chosen based upon experimental results from initial testing and are non-binding for the given networks. Table \ref{tab:computationalresults} summarizes the results.

\begin{table}[htbp]
\centering
\caption{Computational results for problem instances outlined in Table \ref{tab:datasets}}
\label{tab:computationalresults}
\resizebox{\textwidth}{!}{%
\begin{tabular}{cccccccc}
\hline
\textbf{Instance No.} & \textbf{\begin{tabular}[c]{@{}c@{}}No. of \\ Targets\end{tabular}} & \textbf{\begin{tabular}[c]{@{}c@{}}No. of \\ Vehicles\end{tabular}} & \textbf{Coverage} & \textbf{Risk}& \textbf{\begin{tabular}[c]{@{}c@{}}Branch \& Bound\\ Nodes ($x10^3)$\end{tabular}} & \textbf{Gap (\% )} & \textbf{Solve Time (s)}  \\ \hline
1\--1 & 9 & 2 & 92.4 & 85.1 & 1.0 & 0& 68.3  \\
1\--2 & 14 & 2 & 223.2 & 82.9 & 2.7 & 0& 47.8  \\
1\--3 & 12 & 2 & 152.2 & 60.3 & 1.9 & 0& 167.2  \\
1\--4 & 12 & 2 & 124.3 & 72.3 & 2.7 & 0& 39.3  \\
1\--5 & 12 & 2 & 79.2 & 39.5 & 1.6 & 0 & 128.1  \\ \hline
2\--1 & 18 & 7 & 780.8 & 500.0 & 1420.0 & 0.49& ---  \\
2\--2 & 19 & 7 & 1619.7 & 500.0 & 1158.7 & 0.44& ---  \\
2\--3 & 18 & 7 & 2041.9 & 500.0 & 634.0 & 0.73& ---  \\
2\--4 & 17 & 7 & 476.6 & 189.6 & 937.5 & 3.58& ---  \\
2\--5 & 19 & 7 & 1488.8 & 500.0 & 1120.7 & 1.27& ---  \\ \hline
\end{tabular}%
}
\end{table}

All instances from Network 1 were solved to optimality within a few minutes; however, Network 2's instances terminated upon reaching the pre\--defined time limit. While most instances have an optimality gap of less than $1\%$, instance 2\--4 is an exception. Unlike the other Network 2 instances, the risk threshold for instance 2\--4 is not binding upon termination. In fact, no feasible solution exists with incurred risk greater than 200 units. By directing Gurobi to focus on the solution's lower bound (i.e., set \textit{mipfocus} = 3) for instance 2\--4, the optimality gap decreases to 1.4\% in the same 3600 second time limit. 
Further, Table \ref{tab:computationalresults} confirms the number of branch\--and\--bound nodes explored by Gurobi is far greater for Network 2's instances than for Network 1's instances, which is due in part to both network structure and fleet size.

As outlined in Section \ref{sec:sm}, we can approximate the Pareto frontier 
to illustrate the tradeoff between our two competing objective functions (Eqs. \eqref{netcoverage} and \eqref{netrisk}). Without loss of generality, we generate an approximation to the Pareto curve for instance 1\--2 in Figure \ref{fig:47_Pareto} by setting $\delta = 10^{-1}$ (doing so for the remaining instances would show similar insights). 
By choosing a smaller value of $\delta$, a more accurate representation of the true efficient frontier would result; however, due to the continuous nature of the risk function, it is impossible to generate the full Pareto frontier exactly. This methodology will produce dominated and/or weakly-dominated solutions (which can be removed from the frontier) when energy limits and time windows are more restrictive.

\begin{figure}[htbp]
\begin{center}
\includegraphics[width=0.45\textwidth]{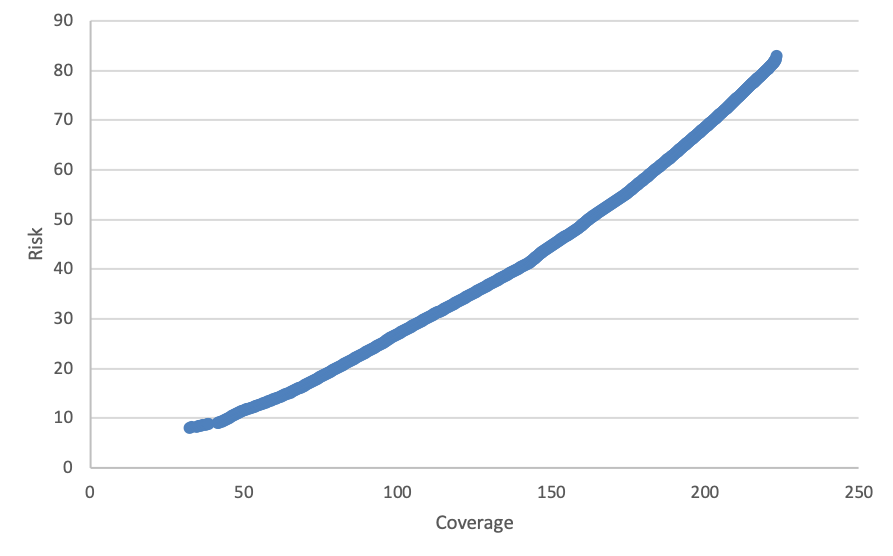}
{\caption{Approximate Pareto frontier for instance 1\--2}\label{fig:47_Pareto}}
\end{center}
\end{figure}

\subsection{Lagrangian Relaxation Based Approach for Special Cases}\label{sec:compresults}
We demonstrate our specialized algorithm (Section \ref{sec:langrangian}) for solving the Lagrangian relaxation problem~\eqref{relaxation} (special case I) and  \eqref{case2model} (special case II) by comparing the Lagrangian dual bound with the solution obtained by Gurobi for three sets of instances: 1) small ($|V^0| = 9, |W|\approx 10, \bar{\eta} =10$), 2) medium ($|V^0| = 12, |W|\approx 11, \bar{\eta} =20$), and 3) large ($|V^0| = 15, |W|\approx 12, \bar{\eta} =20$). All instances are generated using a similar approach to those discussed in Section \ref{sec:results} and utilize a single depot and one vehicle. 

\subsubsection{Special Case I}\label{sec:case1results}
Total operational time is defined as $T := |A|\times\max\{\mathbf{d}(i,j) | (i,j)\in A\}$.  After removing \eqref{netrisk}, \eqref{drt}, \eqref{energy},  \eqref{speedbdv} and \eqref{timewindowslb} from model \eqref{fullmodel}, we set $\Mbar_{ij} \leftarrow T,\ \forall(i,j)\in A$ and $b_i\leftarrow T,\ \forall i\in V$, 
and solve the reduced model using Gurobi 8.1.1.
On average, the Lagrangian dual bounds are $3.9\%$ tighter than the initial Lagrangian relaxation solutions (for the small\-- and medium\--sized instances), with instance $s$-4 seeing the largest improvement of $12.4\%$ (Table \ref{tab:GurobiDPresultsCaseI}). The initial bounds yield a gap of $4.2\%$ from the best solution obtained by Gurobi, whereas the Lagrangian dual bounds are very tight on our test instances -- less than $0.1\%$ of Gurobi's best solution. These bounds, moreover, are obtained by applying our proposed specialized algorithm before Gurobi finishes. We can make marginal improvements to those bounds considering the time required to obtain the Lagrangian dual bound, especially as instance size grows. For large\--scale instances, we do not improve upon the initial Lagrangian relaxation bound within the time limit (an average gap of $1.6\%$ from the best solution obtained by Gurobi) as not many iterations have finished processing due to the complexity for solving a single iteration. However, the proposed algorithm will eventually achieve a good bound ($<0.1\%$ difference) in the absence of a time limit (Table \ref{tab:GurobiDPresultsCaseI_large}).

\begin{table}[htbp]
\centering
\caption{Summary of computation results for both Gurobi and the Lagrangian relaxation based approach for special case I}
\label{tab:GurobiDPresultsCaseI}
\resizebox{\textwidth}{!}{%
\begin{tabular}{c|cccc|ccccc}
\hline
 & \multicolumn{4}{c}{Gurobi} & \multicolumn{5}{|c}{Lagrangian relaxation based approach} \\
\hline
 Instance & \textbf{Coverage} & \begin{tabular}[c]{@{}c@{}}\textbf{Branch \& Bound} \\ \textbf{Nodes} ($\text{x}10^3$)\end{tabular} & \textbf{Gap} (\%) & \begin{tabular}[c]{@{}c@{}}\textbf{Solve} \\ \textbf{Time} \\(s) \end{tabular} & \textbf{\begin{tabular}[c]{@{}c@{}}Initial \\ Lagrangian \\ Relaxation \\ Bound\end{tabular}} & \textbf{\begin{tabular}[c]{@{}c@{}}Lagrangian\\ Dual\\ Bound\end{tabular}} & \textbf{\begin{tabular}[c]{@{}c@{}}Level Bundle \\ Method\\ Iterations\end{tabular}} & \begin{tabular}[c]{@{}c@{}}\textbf{Initial}\\ \textbf{Lagrangian} \\ \textbf{Relaxation}\\ \textbf{Solve Time} (s)\end{tabular} & \begin{tabular}[c]{@{}c@{}}\textbf{Lagrangian} \\ \textbf{Dual Solve}\\ \textbf{Time} (s)\end{tabular} \\ \hline
$s$-0 & 3284.8 & 4119.1 & 0.0 & 285.8 & 3357.0 & 3284.8 & 25 &  0.1 & 0.6    \\
$s$-1 & 6769.0 & 4437.5 & 0.0 & 412.2 & 7037.4 & 6770.6 & 33 & 0.0 & 0.4 \\
$s$-2 & 2622.1 & 7463.0 & 0.0 & 628.3 & 2716.6 & 2622.1 & 23 & 0.0 & 0.3  \\
$s$-3 & 2014.8 & 7357.5 & 0.0 & 554.1  & 2084.2 & 2015.1 & 24  & 0.1 & 0.6 \\
$s$-4 & 2664.1 & 6012.5 & 0.0 & 466.6 & 3041.9 & 2665.5 & 27 & 0.0 & 0.7 \\ \hline
$m$-0 & 8636.5 & 63974.5 & 84.7 & --- & 8954.5 & 8636.5 & 32 & 2.6 & 72.6 \\
$m$-1 & 3480.1 & 67970.2 & 321.0 & --- & 3541.9 & 3480.1  & 22 & 3.8 & 33.5  \\
$m$-2 & 3889.0 & 70868.4 & 180.0 & ---  & 4020.5 & 3892.1 & 28 & 5.7 & 99.9  \\
$m$-3 & 6433.2 & 67861.5 & 188.6 & ---& 6545.1 & 6433.7  & 24 & 10.4 & 124.5  \\
$m$-4 & 3057.8 & 68534.2 & 266.1 & --- & 3171.3 & 3057.8  & 22 & 6.3 & 69.4 \\ \hline
$l$-0 & 12811.0 & 36545.6 & 244.9 & --- & 13163.6 & 13163.6  & 4 & 4123.1 & --- \\
$l$-1 & 20511.8 & 37970.3 & 251.9 & --- & 20608.8 & 20608.8 & 4 & 3547.9 & --- \\
$l$-2 & 11658.4 & 34604.1 & 303.8 & --- & 11771.7 & 11771.7& 10 & 1409.9 & ---   \\
$l$-3 & 12006.8 & 34072.5 & 226.2 & --- & 12316.6 & 12316.6 & 5 & 2071.4 & ---   \\
$l$-4 & 22901.4 & 36217.2 & 180.4 & --- & 23182.9 & 23182.9 & 2  & 4928.0 & ---  \\ \hline
\end{tabular}%
}
\end{table}

\begin{table}[htbp]
\centering
\caption{ Lagrangian relaxation based results for special case I large instances}
\label{tab:GurobiDPresultsCaseI_large}
\begin{tabular}{c|ccccc}
\hline
\textbf{Instance}  & \textbf{\begin{tabular}[c]{@{}c@{}}Initial \\ Lagrangian \\ Relaxation \\ Bound\end{tabular}} & \textbf{\begin{tabular}[c]{@{}c@{}}Lagrangian\\ Dual\\ Bound\end{tabular}} & \textbf{\begin{tabular}[c]{@{}c@{}}Level Bundle \\ Method\\ Iterations\end{tabular}} &  \begin{tabular}[c]{@{}c@{}}\textbf{Initial}\\ \textbf{Lagrangian} \\ \textbf{Relaxation}\\ \textbf{Solve Time} (s)\end{tabular} & \begin{tabular}[c]{@{}c@{}}\textbf{Lagrangian} \\ \textbf{Dual Solve}\\ \textbf{Time} (s)\end{tabular} \\ \hline
$l$-0 & 13163.6 & 12824.2  & 27 & 3187.0 & 29791.6 \\
$l$-1 & 20608.8 & 20513.7& 32 &  3392.6 & 56009.9  \\
$l$-2  & 11771.7 & 11659.4& 32  &  1369.1 & 35070.9 \\
$l$-3 & 12316.6 & 12009.9& 36 &  1946.8 & 41699.6  \\
$l$-4 & 23182.9 & 22901.6& 28 &  4595.4 & 81223.0  \\ \hline
\end{tabular}%
\end{table}

\subsubsection{Special Case II}
Total operational time is defined as $T := |A|\times\max\{\mathbf{d}(i,j) | (i,j)\in A\} \times 0.1$ for all instances, except for $s$-4 where $T := |A|\times\max\{\mathbf{d}(i,j) | (i,j)\in A\} \times 0.2$ as the former deadline causes infeasibility.  We solve the same reduced model outlined in Section \ref{sec:case1results} using Gurobi 8.1.1 with the addition of the following constraint:
\begin{equation}
\sum_{(i,j)\in A} t_{ij} \le T.
\end{equation}
\begin{table}[ht]
\centering
\caption{Summary of computation results for both Gurobi and the Lagrangian relaxation based approach for special case II}
\label{tab:GurobiDPresultsCaseII}
\resizebox{\textwidth}{!}{%
\begin{tabular}{c|cccc|ccccc}
\hline
 & \multicolumn{4}{c}{Gurobi} & \multicolumn{5}{|c}{Lagrangian relaxation based approach} \\
\hline
 Instance & \textbf{Coverage} & \begin{tabular}[c]{@{}c@{}}\textbf{Branch \& Bound} \\ \textbf{Nodes} ($\text{x}10^3$)\end{tabular} & \textbf{Gap} (\%) & \begin{tabular}[c]{@{}c@{}}\textbf{Solve} \\ \textbf{Time} \\(s) \end{tabular} & \textbf{\begin{tabular}[c]{@{}c@{}}Initial \\ Lagrangian \\ Relaxation \\ Bound\end{tabular}} & \textbf{\begin{tabular}[c]{@{}c@{}}Lagrangian\\ Dual\\ Bound\end{tabular}} & \textbf{\begin{tabular}[c]{@{}c@{}}Level Bundle \\ Method\\ Iterations\end{tabular}} & \begin{tabular}[c]{@{}c@{}}\textbf{Initial}\\ \textbf{Lagrangian} \\ \textbf{Relaxation}\\ \textbf{Solve Time} (s)\end{tabular} & \begin{tabular}[c]{@{}c@{}}\textbf{Lagrangian} \\ \textbf{Dual Solve}\\ \textbf{Time} (s)\end{tabular} \\ \hline
$s$-0 & 642.0 & 885.1 & 0.0 & 62.3 & 714.2 & 642.0& 25 & 3.1 & 46.2  \\
$s$-1 & 533.9 & 3030.2 & 0.0 & 277.4 & 802.4 & 535.5 & 28 & 0.4 & 6.9  \\
$s$-2 & 185.7 & 6419.7 & 0.0 & 529.7 & 280.2 & 185.7& 23 & 0.4 & 8.0   \\
$s$-3 & 279.7 & 3823.1 & 0.0 & 327.1 & 351.1 & 280.6& 27 & 3.2 & 34.1  \\
$s$-4 & 318.1 & 4390.6 & 0.0 & 572.0 & 695.8 & 319.3 & 26 & 0.6 & 13.8  \\ \hline
$m$-0 & 597.2 & 62646.2 & 1.7 & --- & 915.3 & 597.2 & 25 & 223.5 & 2521.9  \\
$m$-1 & 353.9 & 61481.7 & 3.2 & --- & 415.7 & 353.9 & 23 & 198.8 & 2548.9 \\
$m$-2 & 672.6 & 69645.8 & 0.9 & --- & 802.1 & 673.2& 29 & 211.1 & 4550.0   \\
$m$-3 & 651.2 & 68185.4 & 1.9 & --- & 763.1 & 659.5 & 10 & 573.8 & --- \\
$m$-4 & 337.0 & 65921.3 & 2.3 & --- & 450.5 & 337.0& 22 & 170.9 & 2054.6  \\ \hline
\end{tabular}}
\end{table}\\
On average, the Lagrangian dual bounds are $26.0\%$ tighter than the initial Lagrangian relaxation solutions, with instance $s$-4 again seeing the largest improvement of $54.1\%$ (Table \ref{tab:GurobiDPresultsCaseII}). The initial bounds yield a gap of $40.8\%$ from the best solution obtained by Gurobi, whereas the Lagrangian dual bounds are within an average of $0.2\%$ of the best Gurobi solution. These bounds, moreover, are obtained by applying our proposed specialized algorithm before Gurobi finishes (except for instance $m$-3). We do not include results for the large\--sized instances as the specialized algorithm does not complete a single iteration within the 7200 second time limit. This is expected, though, as the algorithm crucially depends on the state space, which increases rapidly as each label extension generates two distinct labels.

\section{Conclusions and Future Research} 
\label{sec:conc}

We have proposed and studied the multi\--vehicle covering tour problem with time windows (MCTPTW), which concerns unmanned vehicle routing with energy, time window, and coverage constraints. Motivated by applications for homeland security and public safety, our model aims to find an optimal route for each vehicle in a fleet through a set of secured waypoints in order to surveil a set of targets. We have formulated the problem as a mixed\--integer second\--order cone program with both coverage and risk objectives and test its performance using several benchmark instances. We have considered special cases where the fleet's energy supply and the time windows associated with each waypoint can be approximated by an operational deadline, and where vehicles are not subject to risk, in order to to develop a label correcting algorithm with an innovative set of dominance rules to solve the resulting Lagrangian relaxation problem given a set of dual multipliers. 
Computational experiments examining the efficacy and efficiency of the labeling algorithm show our methodology can produce high quality bounds in the same or less time than a commercial solver. 

For future study, we plan on improving the effectiveness of the labeling algorithm for large\--scale instances. Additionally, we aim to extend the model to manage non-stationary targets or targets whose locations are probabilistic. Other research extensions include the presence of time\--sensitive (on\--demand) targets. Similarly, introducing stochastic time windows or the potential loss of vehicles/waypoints would mimic various levels of uncertainty faced by surveillance operations. These extensions could better assist mission planners to adapt to changing environments within friendly and/or enemy territory.

\bibliography{MCTPTW}%

\begin{thebibliography}{57}
\providecommand{\natexlab}[1]{#1}
\providecommand{\url}[1]{\texttt{#1}}
\expandafter\ifx\csname urlstyle\endcsname\relax
  \providecommand{\doi}[1]{doi: #1}\else
  \providecommand{\doi}{doi: \begingroup \urlstyle{rm}\Url}\fi

\bibitem[Alotaibi et~al.(2018)Alotaibi, Rosenberger, Mattingly, Punugu, and
  Visoldilokpun]{Alotaibi2018}
K.~A. Alotaibi, J.~M. Rosenberger, S.~P. Mattingly, R.~K. Punugu, and
  S.~Visoldilokpun.
\newblock Unmanned aerial vehicle routing in the presence of threats.
\newblock \emph{Computers \& Industrial Engineering}, 115:\penalty0 190 -- 205,
  2018.

\bibitem[Baldacci et~al.(2012)Baldacci, Mingozzi, and Roberti]{Baldacci2012}
R.~Baldacci, A.~Mingozzi, and R.~Roberti.
\newblock Recent exact algorithms for solving the vehicle routing problem under
  capacity and time window constraints.
\newblock \emph{European Journal of Operational Research}, 218\penalty0
  (1):\penalty0 1 -- 6, 2012.

\bibitem[Bard et~al.(2002)Bard, Kontoravdis, and Yu]{Bard2002}
J.~F. Bard, G.~Kontoravdis, and G.~Yu.
\newblock A branch-and-cut procedure for the vehicle routing problem with time
  windows.
\newblock \emph{Transportation Science}, 36\penalty0 (2):\penalty0 250--269,
  2002.

\bibitem[Br{\"a}ysy and Gendreau(2005{\natexlab{a}})]{Braysy2005a}
O.~Br{\"a}ysy and M.~Gendreau.
\newblock Vehicle routing problem with time windows, part i: Route construction
  and local search algorithms.
\newblock \emph{Transportation Science}, 39\penalty0 (1):\penalty0 104--118,
  2005{\natexlab{a}}.

\bibitem[Br{\"a}ysy and Gendreau(2005{\natexlab{b}})]{Braysy2005b}
O.~Br{\"a}ysy and M.~Gendreau.
\newblock Vehicle routing problem with time windows, part ii: Metaheuristics.
\newblock \emph{Transportation Science}, 39\penalty0 (1):\penalty0 119--139,
  2005{\natexlab{b}}.

\bibitem[Cao et~al.(2017)Cao, Wei, Bai, and Qiao]{Cao2017}
Y.~Cao, W.~Wei, Y.~Bai, and H.~Qiao.
\newblock Multi-base multi-uav cooperative reconnaissance path planning with
  genetic algorithm.
\newblock \emph{Cluster Computing}, Aug 2017.

\bibitem[Casbeer et~al.(2006)Casbeer, Kingston, Beard, and McLain]{Casbeer2006}
D.~W. Casbeer, D.~B. Kingston, R.~W. Beard, and T.~W. McLain.
\newblock Cooperative forest fire surveillance using a team of small unmanned
  air vehicles.
\newblock \emph{International Journal of Systems Science}, 37\penalty0
  (6):\penalty0 351--360, 2006.

\bibitem[Chow(2016)]{Chow2016}
J.~Y. Chow.
\newblock Dynamic uav-based traffic monitoring under uncertainty as a
  stochastic arc-inventory routing policy.
\newblock \emph{International Journal of Transportation Science and
  Technology}, 5\penalty0 (3):\penalty0 167 -- 185, 2016.
\newblock Unmanned Aerial Vehicles and Remote Sensing.

\bibitem[Chowdhury et~al.(2017)Chowdhury, Emelogu, Marufuzzaman, Nurre, and
  Bian]{Chowdhury2017}
S.~Chowdhury, A.~Emelogu, M.~Marufuzzaman, S.~G. Nurre, and L.~Bian.
\newblock Drones for disaster response and relief operations: A continuous
  approximation model.
\newblock \emph{International Journal of Production Economics}, 188:\penalty0
  167 -- 184, 2017.

\bibitem[Chung et~al.(2011)Chung, Polak, Royset, and Sastry]{Chung2011}
H.~Chung, E.~Polak, J.~O. Royset, and S.~Sastry.
\newblock On the optimal detection of an underwater intruder in a channel using
  unmanned underwater vehicles.
\newblock \emph{Naval Research Logistics}, 58\penalty0 (8):\penalty0 804 --
  820, 2011.

\bibitem[Coelho et~al.(2017)Coelho, Coelho, Coelho, Ochi, K., Zuidema, Lima,
  and da~Costa]{Coelho2017}
B.~N. Coelho, V.~N. Coelho, I.~M. Coelho, L.~S. Ochi, R.~H. K., D.~Zuidema,
  M.~S. Lima, and A.~R. da~Costa.
\newblock A multi-objective green uav routing problem.
\newblock \emph{Computers \& Operations Research}, 88:\penalty0 306 -- 315,
  2017.

\bibitem[Cordeau et~al.(2000)Cordeau, Desaulniers, Desrosiers, Solomon, and
  Soumis]{Cordeau2000}
J.-F. Cordeau, G.~Desaulniers, J.~Desrosiers, M.~M. Solomon, and F.~Soumis.
\newblock \emph{The {VRP} with time windows}.
\newblock Montr{\'e}al: Groupe d'{\'e}tudes et de recherche en analyse des
  d{\'e}cisions, 2000.

\bibitem[Cordeau et~al.(2007)Cordeau, Laporte, Savelsbergh, and
  Vigo]{Cordeau2007}
J.-F. Cordeau, G.~Laporte, M.~W. Savelsbergh, and D.~Vigo.
\newblock Chapter 6 vehicle routing.
\newblock In C.~Barnhart and G.~Laporte, editors, \emph{Transportation},
  volume~14 of \emph{Handbooks in Operations Research and Management Science},
  pages 367 -- 428. Elsevier, 2007.

\bibitem[Coutinho et~al.(2016)Coutinho, Nascimento, Pessoa, and
  Subramanian]{Coutinho2016}
W.~P. Coutinho, R.~Q.~d. Nascimento, A.~A. Pessoa, and A.~Subramanian.
\newblock A branch-and-bound algorithm for the close-enough traveling salesman
  problem.
\newblock \emph{INFORMS Journal on Computing}, 28\penalty0 (4):\penalty0
  752--765, 2016.

\bibitem[Current and Schilling(1989)]{Current1989}
J.~R. Current and D.~A. Schilling.
\newblock The covering salesman problem.
\newblock \emph{Transportation Science}, 23\penalty0 (3):\penalty0 208--213,
  1989.

\bibitem[Dantzig(1957)]{Dantzig1957}
G.~B. Dantzig.
\newblock Discrete variable extremum problems.
\newblock \emph{Operations Research}, 5:\penalty0 266--277, 1957.

\bibitem[Dantzig and Ramser(1959)]{Dantzig1959}
G.~B. Dantzig and J.~H. Ramser.
\newblock The truck dispatching problem.
\newblock \emph{Management Science}, 6\penalty0 (1):\penalty0 80--91, 1959.

\bibitem[Desaulniers(2010)]{Desaulniers2010}
G.~Desaulniers.
\newblock Branch-and-price-and-cut for the split-delivery vehicle routing
  problem with time windows.
\newblock \emph{Operations Research}, 58\penalty0 (1):\penalty0 179--192, 2010.

\bibitem[Ehrgott(2005)]{Ehrgott2005}
M.~Ehrgott.
\newblock \emph{Multicriteria Optimization}.
\newblock Springer Berlin / Heidelberg, Germany, second edition, 2005.

\bibitem[Ehsani et~al.(2018)Ehsani, Gao, Longo, and Ebrahimi]{Ehsani2018}
M.~Ehsani, Y.~Gao, S.~Longo, and K.~Ebrahimi.
\newblock \emph{Modern Electric, Hybrid Electric, and Fuel Cell Vehicles}.
\newblock CRC Press, Boca Raton, 2018.

\bibitem[Fai{\c c}al et~al.(2014)Fai{\c c}al, Costa, Pessin, Ueyama, Freitas,
  Colombo, Fini, Villas, Os{\'o}rio, Vargas, and Braun]{Faical2014}
B.~S. Fai{\c c}al, F.~G. Costa, G.~Pessin, J.~Ueyama, H.~Freitas, A.~Colombo,
  P.~H. Fini, L.~Villas, F.~S. Os{\'o}rio, P.~A. Vargas, and T.~Braun.
\newblock The use of unmanned aerial vehicles and wireless sensor networks for
  spraying pesticides.
\newblock \emph{Journal of Systems Architecture}, 60\penalty0 (4):\penalty0 393
  -- 404, 2014.

\bibitem[Feillet et~al.(2004)Feillet, Dejax, Gendreau, and
  Gueguen]{Feillet2004}
D.~Feillet, P.~Dejax, M.~Gendreau, and C.~Gueguen.
\newblock An exact algorithm for the elementary shortest path problem with
  resource constraints: Application to some vehicle routing problems.
\newblock \emph{Networks}, 44\penalty0 (3):\penalty0 216 -- 229, 2004.

\bibitem[Fukasawa et~al.(2006)Fukasawa, Longo, Lysgaard, Arag\~{a}o, Reis,
  Uchoa, and Werneck]{Fukasawa2006}
R.~Fukasawa, H.~Longo, J.~Lysgaard, M.~P.~d. Arag\~{a}o, M.~Reis, E.~Uchoa, and
  R.~F. Werneck.
\newblock Robust branch-and-cut-and-price for the capacitated vehicle routing
  problem.
\newblock \emph{Mathematical Programming}, 106\penalty0 (3):\penalty0 491--511,
  May 2006.

\bibitem[Gendreau et~al.(1997)Gendreau, Laporte, and Semet]{Gendreau1997}
M.~Gendreau, G.~Laporte, and F.~Semet.
\newblock The covering tour problem.
\newblock \emph{Operations Research}, 45\penalty0 (4):\penalty0 568--576, 1997.

\bibitem[H\'{a} et~al.(2013{\natexlab{a}})H\'{a}, Bostel, Langevin, and
  Rousseau]{Ha2013a}
M.~H. H\'{a}, N.~Bostel, A.~Langevin, and L.~Rousseau.
\newblock Solving the close‐enough arc routing problem.
\newblock \emph{Networks}, 63:\penalty0 107--118, 2013{\natexlab{a}}.

\bibitem[H\'{a} et~al.(2013{\natexlab{b}})H\'{a}, Bostel, Langevin, and
  Rousseau]{Ha2013b}
M.~H. H\'{a}, N.~Bostel, A.~Langevin, and L.-M. Rousseau.
\newblock An exact algorithm and a metaheuristic for the multi-vehicle covering
  tour problem with a constraint on the number of vertices.
\newblock \emph{European Journal of Operational Research}, 226\penalty0
  (2):\penalty0 211 -- 220, 2013{\natexlab{b}}.

\bibitem[Hachicha et~al.(2000)Hachicha, Hodgson, Laporte, and
  Semet]{Hachicha2000}
M.~Hachicha, M.~J. Hodgson, G.~Laporte, and F.~Semet.
\newblock Heuristics for the multi-vehicle covering tour problem.
\newblock \emph{Computers Operations Research}, 27\penalty0 (1):\penalty0
  29--42, 2000.

\bibitem[Haimes et~al.(1971)Haimes, Lasdon, and Wismer]{Haimes1971}
Y.~Haimes, L.~Lasdon, and D.~Wismer.
\newblock On a bicriterion formulation of the problems of integrated system
  indetification and system optimization.
\newblock \emph{IEEE Transactions on Systems, Man, and Cybernetics},
  1:\penalty0 296--297, 1971.

\bibitem[Herwitz et~al.(2004)Herwitz, Johnson, Dunagan, Higgins, Sullivan,
  Zheng, Lobitz, Leung, Gallmeyer, Aoyagi, Slye, and Brass]{Herwitz2004}
S.~Herwitz, L.~Johnson, S.~Dunagan, R.~Higgins, D.~Sullivan, J.~Zheng,
  B.~Lobitz, J.~Leung, B.~Gallmeyer, M.~Aoyagi, R.~Slye, and J.~Brass.
\newblock Imaging from an unmanned aerial vehicle: agricultural surveillance
  and decision support.
\newblock \emph{Computers and Electronics in Agriculture}, 44\penalty0
  (1):\penalty0 49 -- 61, 2004.

\bibitem[Jozefowiez(2015)]{Jozefowiez2015}
N.~Jozefowiez.
\newblock A branch-and-price algorithm for the multi-vehicle covering tour
  problem.
\newblock \emph{Networks}, 64\penalty0 (3):\penalty0 160--168, 2015.

\bibitem[Kallehauge(2008)]{Kallehauge2008}
B.~Kallehauge.
\newblock Formulations and exact algorithms for the vehicle routing problem
  with time windows.
\newblock \emph{Computers \& Operations Research}, 35\penalty0 (7):\penalty0
  2307 -- 2330, 2008.
\newblock Part Special Issue: Includes selected papers presented at the ECCO'04
  European Conference on combinatorial Optimization.

\bibitem[Kallehauge et~al.(2005)Kallehauge, Larsen, Madsen, and
  Solomon]{Kallehauge2005}
B.~Kallehauge, J.~Larsen, O.~B. Madsen, and M.~M. Solomon.
\newblock Vehicle routing problem with time windows.
\newblock In G.~Desaulniers, J.~Desrosiers, and M.~M. Solomon, editors,
  \emph{Column Generation}, pages 67--98. Springer US, Boston, MA, 2005.

\bibitem[Kanistras et~al.(2015)Kanistras, Martins, Rutherford, and
  Valavanis]{Kanistras2015}
K.~Kanistras, G.~Martins, M.~J. Rutherford, and K.~P. Valavanis.
\newblock Survey of unmanned aerial vehicles (uavs) for traffic monitoring.
\newblock In K.~P. Valavanis and G.~J. Vachtsevanos, editors, \emph{Handbook of
  Unmanned Aerial Vehicles}, pages 2643--2666. Springer Netherlands, Dordrecht,
  2015.

\bibitem[Laporte(2007)]{Laporte2007}
G.~Laporte.
\newblock What you should know about the vehicle routing problem.
\newblock \emph{Naval Research Logistics}, 54\penalty0 (8):\penalty0 811--819,
  2007.

\bibitem[Lemar\'{e}chal et~al.(1995)Lemar\'{e}chal, Nemirovskii, and
  Nesterov]{Lemarechal1995}
C.~Lemar\'{e}chal, A.~Nemirovskii, and Y.~Nesterov.
\newblock New variants of bundle methods.
\newblock \emph{Mathematical Programming}, 69:\penalty0 111--147, 1995.

\bibitem[Liu et~al.(2016)Liu, Zhang, Yu, and Yuan]{Liu2016}
Z.~Liu, Y.~Zhang, X.~Yu, and C.~Yuan.
\newblock Unmanned surface vehicles: {A}n overview of developments and
  challenges.
\newblock \emph{Annual Reviews in Control}, 41:\penalty0 71 -- 93, 2016.

\bibitem[Lysgaard et~al.(2004)Lysgaard, Letchford, and Eglese]{Lysgaard2004}
J.~Lysgaard, A.~N. Letchford, and R.~W. Eglese.
\newblock A new branch-and-cut algorithm for the capacitated vehicle routing
  problem.
\newblock \emph{Mathematical Programming}, 100\penalty0 (2):\penalty0 423--445,
  Jun 2004.

\bibitem[Murphy et~al.(2008)Murphy, Steimle, Griffin, Cullins, Hall, and
  Pratt]{Murphy2008}
R.~R. Murphy, E.~Steimle, C.~Griffin, C.~Cullins, M.~Hall, and K.~Pratt.
\newblock Cooperative use of unmanned sea surface and micro aerial vehicles at
  hurricane wilma.
\newblock \emph{Journal of Field Robotics}, 25\penalty0 (3):\penalty0 164--180,
  3 2008.

\bibitem[Olagundoye(1971)]{Olag71}
O.~B. Olagundoye.
\newblock \emph{Efficiency and the [Epsilon]-constraint Approach for
  Multi-criterion Systems}.
\newblock PhD thesis, Case Western Reserve University, 1971.

\bibitem[Otto et~al.(2018)Otto, Agatz, Campbell, Golden, and Pesch]{Otto2018}
A.~Otto, N.~Agatz, J.~Campbell, B.~Golden, and E.~Pesch.
\newblock Optimization approaches for civil applications of unmanned aerial
  vehicles ({UAVs}) or aerial drones: {A} survey.
\newblock \emph{Networks}, 72\penalty0 (4):\penalty0 411--458, 2018.

\bibitem[Parragh et~al.(2008{\natexlab{a}})Parragh, Doerner, and
  Hartl]{Parragh2008a}
S.~N. Parragh, K.~F. Doerner, and R.~F. Hartl.
\newblock A survey on pickup and delivery problems. {}part i: {T}ransportation
  between customers and depot.
\newblock \emph{Journal f{\"u}r Betriebswirtschaft}, 58\penalty0 (1):\penalty0
  21--51, Apr 2008{\natexlab{a}}.

\bibitem[Parragh et~al.(2008{\natexlab{b}})Parragh, Doerner, and
  Hartl]{Parragh2008b}
S.~N. Parragh, K.~F. Doerner, and R.~F. Hartl.
\newblock A survey on pickup and delivery problems. {P}art ii: {T}ransportation
  between pickup and delivery locations.
\newblock \emph{Journal f{\"u}r Betriebswirtschaft}, 58\penalty0 (1):\penalty0
  81--117, Apr 2008{\natexlab{b}}.

\bibitem[Poikonen et~al.(2017)Poikonen, Wang, and Golden]{Poikonen2017}
S.~Poikonen, X.~Wang, and B.~Golden.
\newblock The vehicle routing problem with drones: Extended models and
  connections.
\newblock \emph{Networks}, 70:\penalty0 34--43, 2017.

\bibitem[Potvin(2009)]{Potvin2009}
J.-Y. Potvin.
\newblock State-of-the art review---evolutionary algorithms for vehicle
  routing.
\newblock \emph{INFORMS Journal on Computing}, 21\penalty0 (4):\penalty0
  518--548, 2009.

\bibitem[Ropke and Cordeau(2009)]{Ropke2009}
S.~Ropke and J.-F. Cordeau.
\newblock Branch and cut and price for the pickup and delivery problem with
  time windows.
\newblock \emph{Transportation Science}, 43\penalty0 (3):\penalty0 267--286,
  2009.

\bibitem[Shetty et~al.(2008)Shetty, Sudit, and Nagi]{Shetty2008}
V.~K. Shetty, M.~Sudit, and R.~Nagi.
\newblock Priority-based assignment and routing of a fleet of unmanned combat
  aerial vehicles.
\newblock \emph{Computers \& Operations Research}, 35\penalty0 (6):\penalty0
  1813--1828, 2008.
\newblock Part Special Issue: OR Applications in the Military and in
  Counter-Terrorism.

\bibitem[Shuttleworth et~al.(2008)Shuttleworth, Golden, Smith, and
  Wasil]{Shuttleworth2008}
R.~Shuttleworth, B.~L. Golden, S.~Smith, and E.~Wasil.
\newblock Advances in meter reading: Heuristic solution of the close enough
  traveling salesman problem over a street network.
\newblock In B.~Golden, S.~Raghavan, and E.~Wasil, editors, \emph{The Vehicle
  Routing Problem: Latest Advances and New Challenges}, pages 487--501.
  Springer US, Boston, MA, 2008.

\bibitem[Solomon(1987)]{Solomon1987}
M.~M. Solomon.
\newblock Algorithms for the vehicle routing and scheduling problems with time
  window constraints.
\newblock \emph{Operations Research}, 35\penalty0 (2):\penalty0 254--265, 1987.

\bibitem[{\v{S}}vec et~al.(2014){\v{S}}vec, Thakur, Raboin, Shah, and
  Gupta]{Svec2014}
P.~{\v{S}}vec, A.~Thakur, E.~Raboin, B.~C. Shah, and S.~K. Gupta.
\newblock Target following with motion prediction for unmanned surface vehicle
  operating in cluttered environments.
\newblock \emph{Autonomous Robots}, 36\penalty0 (4):\penalty0 383--405, Apr
  2014.

\bibitem[Tokekar et~al.(2016)Tokekar, Hook, Mulla, and Isler]{Tokekar2016}
P.~Tokekar, J.~V. Hook, D.~Mulla, and V.~Isler.
\newblock Sensor planning for a symbiotic uav and ugv system for precision
  agriculture.
\newblock \emph{IEEE Transactions on Robotics}, 32\penalty0 (6):\penalty0
  1498--1511, 2016.

\bibitem[Toth and Vigo(2002)]{Toth2002}
P.~Toth and D.~Vigo.
\newblock \emph{The vehicle routing problem}.
\newblock SIAM, 2002.

\bibitem[Tricoire et~al.(2012)Tricoire, Graf, and Gutjahr]{Tricoire2012}
F.~Tricoire, A.~Graf, and W.~J. Gutjahr.
\newblock The bi-objective stochastic covering tour problem.
\newblock \emph{Computers \& Operations Research}, 39\penalty0 (7):\penalty0
  1582 -- 1592, 2012.

\bibitem[Wallar et~al.(2015)Wallar, Plaku, and Sofge]{Wallar2015}
A.~Wallar, E.~Plaku, and D.~A. Sofge.
\newblock Reactive motion planning for unmanned aerial surveillance of
  risk-sensitive areas.
\newblock \emph{IEEE Transactions on Automation Science and Engineering},
  12\penalty0 (3):\penalty0 969--980, 2015.

\bibitem[Xia et~al.(2017)Xia, Batta, and Nagi]{Xia2017}
Y.~Xia, R.~Batta, and R.~Nagi.
\newblock Controlling a fleet of unmanned aerial vehicles to collect uncertain
  information in a threat environment.
\newblock \emph{Operations Research}, 65\penalty0 (3):\penalty0 674--692, 2017.

\bibitem[Xu et~al.(2010)Xu, Xu, and Li]{Xu2010}
Z.~Xu, L.~Xu, and C.-L. Li.
\newblock Approximation results for min\mbox{-}max path cover problems in
  vehicle routing.
\newblock \emph{Naval Research Logistics}, 57\penalty0 (8):\penalty0 728--748,
  2010.

\bibitem[Zabarankin et~al.(2002)Zabarankin, Uryasev, and
  Pardalos]{Zabarankin2002}
M.~Zabarankin, S.~Uryasev, and P.~Pardalos.
\newblock Optimal risk path algorithms.
\newblock In R.~Murphey and P.~M. Pardalos, editors, \emph{Cooperative Control
  and Optimization}, chapter~1, pages 273--303. Kluwer Academic Publishers,
  2002.

\bibitem[Zabarankin et~al.(2006)Zabarankin, Uryasev, and
  Murphey]{Zabarankin2006}
M.~Zabarankin, S.~Uryasev, and R.~Murphey.
\newblock Aircraft routing under the risk of detection.
\newblock \emph{Naval Research Logistics}, 53\penalty0 (8):\penalty0 728--747,
  12 2006.

\end{thebibliography}

\end{document}